\documentclass{elsarticle}

\journal{European Journal of Combinatorics}

\usepackage{amssymb,amsmath,amsthm}
\usepackage{marvosym}
\usepackage{enumerate}
\usepackage[all]{xy}
\usepackage[utf8]{inputenc}
\usepackage{tikz}
\usetikzlibrary{arrows,fit,positioning,decorations.pathmorphing,decorations.markings}

\theoremstyle{plain}
  \newtheorem*{theorem*}{Theorem}
  \newtheorem{theorem}{Theorem}[section]
  \newtheorem{lemma}[theorem]{Lemma}
  \newtheorem{conjecture}{Conjecture}
  \newtheorem*{conjecture*}{Conjecture}
  
  \newtheorem{corollary}[theorem]{Corollary}
  \newtheorem{observation}[theorem]{Observation}
\theoremstyle{definition}
  \newtheorem{definition}[theorem]{Definition}
  
  \newtheorem*{problem*}{Problem}
\theoremstyle{remark}
  \newtheorem*{remark*}{Remark}
  
  \newtheorem*{claim*}{Claim}

\usepackage{etoolbox}
  \newcounter{claimcount}
  \theoremstyle{remark}
    \newtheorem{claim}[claimcount]{Claim}
    \AtBeginEnvironment{lemma}{\setcounter{claimcount}{0}}

\newenvironment{subproof}[1][\proofname]{%
  \begin{proof}[#1]%
}{%
  \end{proof}%
}

\newcommand{\abs}{\mathrel{\unlhd}}
\newcommand{\compose}{\mathbin{{<}\mkern-14mu{-}}}

\usepackage{centernot}
  \newcommand{\myrightarrow}{\mathrel{\longrightarrow\mkern-13mu{\blacktriangleright}}}

\DeclareMathOperator{\lvl}{\mathrm{lvl}}
\DeclareMathOperator{\hgt}{\mathrm{hgt}}
\DeclareMathOperator{\alg}{\mathbf{alg}}
\DeclareMathOperator{\IdPol}{\mathrm{IdPol}}
\DeclareMathOperator{\Clo}{\mathrm{Clo}}
\DeclareMathOperator{\dist}{\mathrm{dist}}
\DeclareMathOperator{\Path}{\mathrm{Path}}

\begin{document}

\begin{frontmatter}

\title{On the complexity of $\mathbb H$-coloring for special oriented trees}

\author{Jakub Bulín}
\address{Department of Algebra, Faculty of Mathematic and Physics, Charles University, Sokolovsk{\'a}~83, 18675 Prague 8, Czechia\\ }
\ead{jakub.bulin@gmail.com}

\begin{abstract}
For a fixed digraph $\mathbb H$, the $\mathbb H$-coloring problem is
the problem of deciding whether a given input digraph $\mathbb G$ admits a homomorphism to
$\mathbb H$. The CSP dichotomy conjecture of Feder and Vardi is equivalent to proving that, for any $\mathbb H$, the $\mathbb H$-coloring problem is in \textbf{P} or \textbf{NP}-complete. We confirm this dichotomy for a certain class of oriented trees, which we call special trees (generalizing earlier results on special triads and polyads). Moreover, we prove that every tractable special oriented tree has bounded width, i.e., the corresponding $\mathbb H$-coloring problem is solvable by local consistency checking. Our proof relies on recent algebraic tools, namely characterization of congruence meet-semidistributivity via pointing operations and absorption theory.

{\noindent\scriptsize\textit{\copyright 2017. This manuscript version is made available under the CC-BY-NC-ND 4.0 license.}}
\end{abstract}


\end{frontmatter}

\section{Introduction}

The Constraint Satisfaction Problem (CSP) provides a common framework for various problems from theoretical computer science as well as for many real-life applications (e.g. in graph theory, database theory, artificial intelligence, scheduling). Its history dates back to the 1970s and it has been central to the development of theoretical computer science in the past few decades.

For a fixed (finite) relational structure $\mathbb A$, the Constraint Satisfaction Problem with template $\mathbb A$, or $\mathrm{CSP}(\mathbb A)$ for short, is the following decision problem:
\begin{quote}
INPUT: A relational structure $\mathbb X$ (of the same type as $\mathbb A$).\\
QUESTION: Is there a homomorphism from $\mathbb X$ to $\mathbb A$?
\end{quote}
\noindent For a (directed) graph $\mathbb H$, $\mathrm{CSP}(\mathbb H)$ is also commonly referred to as the $\mathbb H$-coloring problem.

A lot of interest in this class of problems was sparked by a seminal work of Feder and Vardi \cite{feder_computational_1999}, in which the authors established a connection to computational complexity theory. They conjectured a large natural class of \textbf{NP} decision problems avoiding the complexity classes strictly between \textbf{P} and \textbf{NP}-complete (assuming that \textbf{P}$\neq$\textbf{NP}): Monotone Monadic SNP without inequality (MMSNP). Many natural decision problems, such as $k$-SAT, graph $k$-colorability or solving systems of linear equations over finite fields belong to this class. They also proved that each problem from this class is polynomial time equivalent to $\mathrm{CSP}(\mathbb A)$, for some relational structure $\mathbb A$. (The reduction from CSP to MMSNP, originally randomized, was later derandomized by Kun in \cite{kun_constraints_2013}.) Hence their conjecture can be formulated as follows.

\begin{conjecture}[The CSP dichotomy conjecture]
For every (finite) relational structure $\mathbb A$, $\mathrm{CSP}(\mathbb A)$ is in \textbf{P} or \textbf{NP}-complete.
\end{conjecture}

At that time this conjecture was supported by two major cases: Schaefer's dichotomy result for two-element domains \cite{schaefer_complexity_1978} and the dichotomy theorem for undirected graphs by Hell and Ne{\v s}et{\v r}il \cite{hell_complexity_1990}. A major breakthrough followed the work of Jeavons, Cohen and Gyssens \cite{jeavons_closure_1997}, later refined by Bulatov, Jeavons and Krokhin \cite{bulatov_classifying_2005}, which uncovered an intimate connection between the constraint satisfaction problem and universal algebra. This connection brought a better understanding of the known results as well as a number of new results which seemed out of reach for the pre-algebraic methods. The most important results include dichotomies for three-element domains \cite{bulatov_dichotomy_2006} and for conservative structures (i.e., containing all subsets as unary relations) \cite{bulatov_tractable_2003} by Bulatov (see also \cite{barto_dichotomy_2011}), a characterization of solvability by the \emph{few subpowers} algorithm (a generalization of Gaussian elimination) by Berman et al \cite{berman_varieties_2010,idziak_tractability_2007} and solvability by local consistency checking (so-called \emph{bounded width}) by Barto and Kozik \cite{barto_constraint_2014} (conjectured in \cite{larose_bounded_2007}). Larose and Tesson \cite{larose_universal_2009} successfully applied the theory to study finer complexity classes of CSPs extending the result of Allender et al \cite{allender_refining_2005} for boolean CSPs.

The connection between CSPs and algebras turned out to be fruitful in both directions; it has lead to a discovery of important structural properties of finite algebras. Of particular importance to us is the theory of absorption by Barto and Kozik \cite{barto_absorbing_2012,barto_maltsev_2013,barto_absorption_2017} and a characterization of congruence meet-semiditributivity via pointing operations by Barto, Kozik, and Stanovsk{\' y} \cite{barto_robust_2012,barto_maltsev_2013}.

In the paper \cite{feder_computational_1999}, Feder and Vardi also constructed, for every structure $\mathbb A$, a directed graph $\mathcal D'(\mathbb A)$ such that $\mathrm{CSP}(\mathbb A)$ and $\mathrm{CSP}(\mathcal D'(\mathbb A))$ are polynomial-time equivalent. Hence the CSP dichotomy conjecture is equivalent to its restriction to digraphs. A streamlined variant of this reduction which we will denote by $\mathcal D(\mathbb A)$ (and which is, in fact, logspace) is studied by the author, Deli{\' c}, Jackson, and Niven in \cite{bulin_reduction_2013,bulin_finer_2015}, where we prove that most properties relevant to the CSP carry over from $\mathbb A$ to $\mathcal D(\mathbb A)$. As a consequence, the algebraic conjectures characterizing CSPs solvable in \textbf{P} \cite{bulatov_classifying_2005}, \textbf{NL}, and \textbf{L} \cite{larose_universal_2009} are equivalent to their restrictions to digraphs. The digraphs $\mathcal D(\mathbb A)$ are, in fact, \emph{special balanced digraphs} in the terminology of this paper, a generalization of special triads, special polyads and special trees discussed below.

Using the algebraic approach, Barto, Kozik, and Niven confirmed the conjecture of Bang-Jensen and Hell and proved dichotomy for \emph{smooth digraphs} \cite{barto_csp_2008} (i.e., digraphs with no sources and no sinks). The dichotomy was also established for a number of other classes of digraphs, e.g. oriented paths \cite{gutjahr_polynomial_1992} (which are all tractable) or oriented cycles \cite{feder_classification_2001}. See \cite{larose_algebra_2017} for a recent survey.

This paper is concerned with the $\mathbb H$-coloring problem for oriented trees. In the class of all digraphs, oriented trees are in some sense very far from smooth digraphs, and the algebraic tools seem to be not yet developed enough to deal with them. Hence oriented trees serve as a good field-test for new methods.

Apart from oriented paths, the simplest class of oriented trees are \emph{triads} (i.e., oriented trees with one vertex of degree $3$ and all other vertices of degree $2$ or $1$); the CSP dichotomy remains open even for triads. Among the triads, Hell, Ne{\v s}et{\v r}il and Zhu \cite{hell_complexity_1996,hell_duality_1996} identified a (fairly restricted) subclass, for which they coined the term \emph{special triads} and which allowed them to handle at least some examples. For instance, they constructed a special triad that gives rise to an \textbf{NP}-complete $\mathbb H$-coloring problem.

In \cite{barto_csp_2009}, Barto et al used algebraic methods to prove that every special triad either gives rise to an \textbf{NP}-complete $\mathbb H$-coloring problem, or possesses a compatible majority operation (so-called \emph{strict width~2}) or compatible totally symmetric idempotent operations of all arities (so-called \emph{width~1}). In \cite{barto_csp_2013}, the author and Barto established the CSP dichotomy conjecture for \emph{special polyads}, a generalization of special triads where the one vertex of degree greater than 2 is allowed to have an arbitrary degree. In particular, every tractable core special polyad has bounded width. However, there are special polyads which have bounded width, but neither bounded strict width nor width~1.

In this paper, we study \emph{special trees}, a broad generalization of special triads and special polyads. We confirm the CSP dichotomy conjecture for special trees and, moreover, prove that every tractable core special tree has bounded width.

The proof uses modern tools from the algebraic approach to the CSP (in particular, absorption and pointing operations \cite{barto_maltsev_2013}) and is somewhat simpler and more natural than the proofs in \cite{barto_csp_2009} and \cite{barto_csp_2013}. Therefore we believe that there is hope for further generalization. In particular, we conjecture that tractability implies bounded width for all oriented trees.

\subsection*{Organization of the paper}

Section~2 introduces basic notions and fixes notation used throughout the paper. In Section~3 we define special trees (and the previously studied subclasses: special triads and special polyads) and state the main result. Section~4 contains the universal algebraic tools we use. The main result is proved in Section~5. In the last section we discuss the results and related open problems.

\section{Preliminaries} \label{section:preliminaries}

In this section, we introduce basic notions and fix notation used throughout the paper. Our aim is to make the paper accessible to a wider audience outside of universal algebra. Therefore we only assume the reader to possess some very basic knowledge of graph theory and universal algebra. We refrain from using specialist terminology wherever possible, or move it to explanatory remarks which the reader may skip.

We recommend \cite{hell_graphs_2004} for a detailed exposition of digraphs, relational structures (under the name ``general relational systems'') and their homomorphisms as well as an introduction to graph coloring and constraint satisfaction. For an introduction to the notions from universal algebra that are not explained in detail in this paper, we invite the reader to consult \cite{bergman_universal_2011}.  The primary source for the algebraic approach to the CSP is the paper \cite{bulatov_classifying_2005}.

\subsection{Notation}

For a positive integer $n$ we denote the set $\{1,2,\dots,n\}$ by
$[n]$; we set $[0]=\emptyset$. We write tuples using boldface notation, e.g. $\mathbf
a=(a_1,a_2,\dots,a_k)\in A^k$. When ranging over
tuples we use superscripts, e.g. $(\mathbf{a^1},\mathbf{a^2},\dots,\mathbf{a^n})\in (A^k)^n$, where $\mathbf{a^i}=(a^i_1,a^i_2,\dots,a^i_k)$, for $i\in [n]$. We sometimes write $\langle a_1a_2\dots\rangle$ to denote a sequence of elements.

\subsection{Relational structures}

A relational signature $\sigma$ is a (in our case finite) set of relation symbols $R_i$, each with an associated arity $k_i$. A (finite) \emph{relational structure} $\mathbb A$ with signature $\sigma$ is a finite, nonempty set $A$ (the \emph{universe}) equipped with relations $R_i^\mathbb A\subseteq A^{k_i}$, for each relation symbol $R_i$ of arity $k_i$ in $\sigma$. We follow the standard convention of using $A$, $B$, \dots\ to denote the universe of $\mathbb A$, $\mathbb B$, \dots

Let $\mathbb A$ and $\mathbb B$ be two $\sigma$-structures.
A mapping $\varphi:A\to B$ is a \emph{homomorphism} from $\mathbb A$ to $\mathbb B$, if for each relational symbol $R_i$ of arity $k_i$ in $\sigma$ and for each $\mathbf a\in R_i^\mathbb A$  we have $(\varphi(a_1),\dots,\varphi(a_k))\in R_i^\mathbb B$.

We write $\varphi:\mathbb A\to\mathbb B$ to mean that $\varphi$ is a homomorphism from $\mathbb A$ to $\mathbb B$, and $\mathbb A\to\mathbb B$ to mean that there exists a homomorphism from $\mathbb A$ to $\mathbb B$.

For every $\mathbb A$ there exists a relational structure $\mathbb A'$ such that $\mathbb A\to\mathbb A'$, $\mathbb A'\to\mathbb A$, and $\mathbb A'$ is of minimal size with respect to these properties; such structure $\mathbb A'$ is called the \emph{core of $\mathbb A$} (it is unique up to isomorphism); $\mathbb A$ is a \emph{core} if it is the core of itself.

We will be almost exclusively interested in a special type of relational structures: \emph{directed graphs}.

\subsection{Digraphs}  A \emph{digraph} (short for ``\emph{directed graph}'') is a relational
structure $\mathbb G=(G;E)$ with a single binary relation $E\subseteq G^2$. We call $u\in G$ and $(u,v)\in E$ (sometimes denoted by $u\myrightarrow v$) \emph{vertices} and \emph{edges} of $\mathbb G$, respectively. A digraph $\mathbb G'=(G';E')$ is a \emph{subgraph} of $\mathbb G$, if $G'\subseteq G$ and $E'\subseteq E$. It is an \emph{induced subgraph} if $E'=E\cap (G')^2$.

An \emph{oriented path} is a digraph $\mathbb P$ which consists of a non-repeating sequence of vertices $\langle v_0 v_1 \dots v_k\rangle$ (allowing for the degenerate case $k=0$) such that precisely one of $(v_{i-1},v_i),(v_i,v_{i-1})$ is an edge, for each $i\in [k]$; the number $k$ is called the \emph{length} of $\mathbb P$. We usually require oriented paths to have a fixed direction, and thus an \emph{initial} and a \emph{terminal} vertex. An \emph{oriented cycle} is a digraph which can be obtained from an oriented path of nonzero length by identifying the initial and terminal vertex.

For $a,b\in G$ we say that $a$ is \emph{connected} to $b$ in $\mathbb G$ via an oriented path $\mathbb P$, if $\mathbb P$ is a subgraph of $\mathbb G$ and $a$ and $b$ are the initial and terminal vertex of $\mathbb P$, respectively. The \emph{distance} between $a$ and $b$ in $\mathbb G$, denoted $\dist_\mathbb G(a,b)$, is then the length of the shortest oriented path $\mathbb P'$ connecting $a$ to $b$ in $\mathbb G$.
Connectivity is an equivalence relation, its classes are \emph{connected components} of $\mathbb G$ and $\mathbb G$ is \emph{connected} if it consists of a single connected component.

For $n>0$, the \emph{$n$th direct power of $\mathbb G$} is the digraph $\mathbb G^n=(G^n,E^n)$, i.e., its vertices are $n$-tuples of vertices of $\mathbb G$ and the edge relation is
$$
\{(\mathbf u,\mathbf v)\in (G^n)^2\mid (u_i,v_i)\in E\text{ for all }i\in[n]\}.
$$
Connectivity in direct powers of digraphs will play an important role.

An \emph{oriented tree} is a connected digraph containing no oriented cycles. Equivalently, it is a digraph in which every two vertices are connected via a unique oriented path. Oriented paths and trees are natural examples of balanced digraphs: a connected digraph is \emph{balanced} if it admits a \emph{level function} $\lvl:G\to\mathbb N\cup\{0\}$, where $\lvl(b)=\lvl(a)+1$ whenever $(a,b)$ is an edge, and the minimum level is $0$. The maximum level is called the \emph{height} of digraph $\mathbb G$ and denoted by $\hgt(\mathbb G)$.

\subsection{Algebras}

A \emph{$k$-ary operation} on a set $A$ is a mapping $f\colon A^k\to A$. By an \emph{algebra} we mean a pair $\mathbf A=(A;\mathcal F)$, where $A$ is a nonempty set and $\mathcal F$ is a set of operations on $A$ (so-called \emph{basic operations} of $\mathbf A$). We denote by $\Clo(\mathbf A)$ the set of all \emph{term operations} of $\mathbf A$ (i.e., operations obtained from $\mathcal F$ together with the \emph{projection} operations by composition).

A subset $B\subseteq A$ is a \emph{subuniverse} of $\mathbf A$ (denoted by $B\leq\mathbf A$) if it is closed under all (basic, or equivalently term) operations of $\mathbf A$. A nonempty subuniverse $B$ is an algebra in its own right, equipped with operations of $\mathbf A$ restricted to $B$, i.e., $(B; \{f |_B\mid f\in\mathcal F\})$. We will frequently use the fact that an intersection of subuniverses is again a subuniverse.

An operation is \emph{idempotent} if $f(x,x,\dots,x)=x$ for all $x\in A$. An algebra is idempotent if all of its (basic, or equivalently term) operations are idempotent. Note that an algebra $\mathbf A$ is idempotent, if and only if $\{a\}\leq\mathbf A$ for every $a\in A$.

For $n>0$, the \emph{$n$th power} of $\mathbf A$ is the algebra $\mathbf A^n=(A^n;\{f\times\dots \times f\mid f\in\mathcal F\})$ where $f\times\dots \times f$ means that $f$ is applied to $n$-tuples of elements coordinatewise.

We write $C\leq B\leq\mathbf A$ to mean that both $B$ and $C$ are subuniverses of $\mathbf A$ and $C\subseteq B$. In particular, if $B$ and $C$ are subuniverses of $\mathbf A$, then $E\leq B\times C$ means that $E$ is a subuniverse of $\mathbf A^2$ contained in $B\times C$ (which is a subuniverse of $\mathbf A^2$ as well).

All algebras we will work with are subuniverses of a certain finite idempotent algebra (or rarely of its 2nd power): the \emph{algebra of idempotent polymorphisms} of some fixed relational structure.
\subsection{Algebra of idempotent polymorphisms}

Note that a digraph homomorphism is simply an edge-preserving mapping. The notion of digraph \emph{polymorphism} is a natural generalization to higher arity operations:

A $k$-ary ($k>0$) operation $\varphi$ on $G$ is a \emph{polymorphism} of a digraph $\mathbb G$, if it is a homomorphism from $\mathbb G^k$ to $\mathbb G$. This means that $\varphi$ preserves edges in the following sense: if $a_i\myrightarrow b_i$ for $i\in[k]$, then $\varphi(\mathbf a)\myrightarrow\varphi(\mathbf b)$. The notions of $k$th direct power, preserving a relation, and polymorphism generalize naturally to relational structures.

Let $\mathbb A$ be a relational structure. The \emph{algebra of idempotent polymorphisms} of $\mathbb A$ is the algebra $\alg\mathbb A=(A;\IdPol(\mathbb A))$, where $\IdPol(\mathbb A)$ denotes the set of all idempotent polymorphisms of $\mathbb A$; we write $\IdPol_k(\mathbb A)$ to denote its $k$-ary part.

A relation $S\subseteq A^n$ is \emph{primitive positive definable} from $\mathbb A$ \emph{with constants}, if it is definable by an existentially quantified conjunction of atomic formul{\ae} of the form $x_i=a$ or $R(x_{i_1},\dots,x_{i_j})$, where $a\in A$ and $R$ is one of the relations of $\mathbb A$. The following fact, based on the Galois correspondence between clones and relational clones \cite{bodnarcuk_galois_1969,geiger_closed_1968} is central to the algebraic approach to the CSP.

\begin{lemma}[see {\cite[Proposition 2.21]{bulatov_classifying_2005}}] \label{lemma:pp-definable=subuniverse}
A relation $S\subseteq A^n$ is primitive positive definable from $\mathbb A$ with constants, if and only if $S$ is a subuniverse of $(\alg\mathbb A)^n$.
\end{lemma}

The connection between universal algebra and constraint satisfaction is discussed in detail in \cite{bulatov_classifying_2005,bulatov_recent_2008,barto_absorption_2017}.

\section{Special trees \& the main result} \label{section:result}

In this section, we define special trees and state the main result of this paper.

\begin{definition}
An oriented path $\mathbb P$ with initial vertex $a$ and terminal vertex $b$ is \emph{minimal} if $\lvl(a)=0$, $\lvl(b)=\hgt(\mathbb P)$, and $0<\lvl(v)<\hgt(\mathbb P)$ for every $v\in P\setminus\{a,b\}$.
\end{definition}

Minimal paths have the property that their \emph{net length} (the number of forward edges minus the number of backward edges) is strictly greater than the net length of any of their subpaths. An example of a minimal path is depicted in Figure \ref{figure:minimal_path} below.

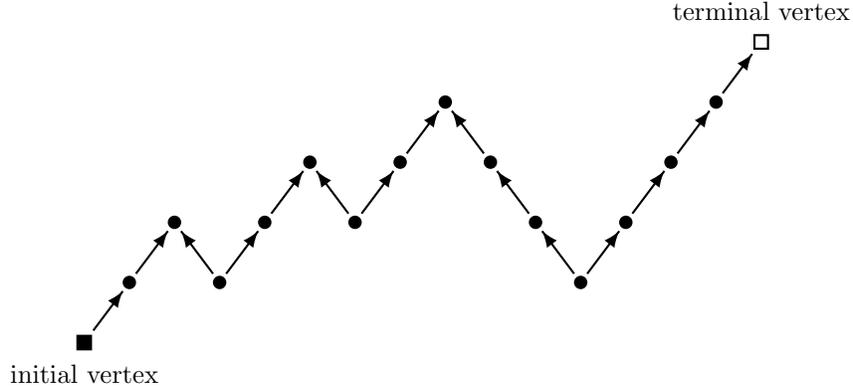
\begin{figure}[ht!]\centering

\begin{tikzpicture}
    [thick,on grid,node distance=.8cm and .6cm,
    dot/.style={circle,draw,outer sep=2pt,inner sep=0pt,minimum size=1.5mm,fill},
    blacksquare/.style={rectangle,draw,outer sep=2pt,inner sep=0pt,minimum size=1.8mm,fill},
    whitesquare/.style={rectangle,draw,outer sep=2pt,inner sep=0pt,minimum size=1.8mm},
    decoration={markings,mark=at position 1 with {\arrow[scale=1.2,black]{latex}};}
]

\node [blacksquare,label={south:\text{initial vertex}}] (a) at (0,0) {};
\node [dot,above right=of a] (b)  {};
\node [dot,above right=of b] (c)  {};
\node [dot,below right=of c] (d)  {};
\node [dot,above right=of d] (e)  {};
\node [dot,above right=of e] (f)  {};
\node [dot,below right=of f] (g)  {};
\node [dot,above right=of g] (h)  {};
\node [dot,above right=of h] (i)  {};
\node [dot,below right=of i] (j)  {};
\node [dot,below right=of j] (k)  {};
\node [dot,below right=of k] (l)  {};
\node [dot,above right=of l] (m)  {};
\node [dot,above right=of m] (n)  {};
\node [dot,above right=of n] (o)  {};
\node [whitesquare,above right=of o,label={north:\text{terminal vertex}}] (p)  {};
\draw[postaction={decorate}] (a) -- (b);
\draw[postaction={decorate}] (b) -- (c);
\draw[postaction={decorate}] (d) -- (c);
\draw[postaction={decorate}] (d) -- (e);
\draw[postaction={decorate}] (e) -- (f);
\draw[postaction={decorate}] (g) -- (f);
\draw[postaction={decorate}] (g) -- (h);
\draw[postaction={decorate}] (h) -- (i);
\draw[postaction={decorate}] (j) -- (i);
\draw[postaction={decorate}] (k) -- (j);
\draw[postaction={decorate}] (l) -- (k);
\draw[postaction={decorate}] (l) -- (m);
\draw[postaction={decorate}] (m) -- (n);
\draw[postaction={decorate}] (n) -- (o);
\draw[postaction={decorate}] (o) -- (p);
\end{tikzpicture}

\caption{A minimal path.}\label{figure:minimal_path}
\end{figure}

We will need the following well-known fact. A proof can be found in \cite{haggkvist_multiplicative_1988}.

\begin{lemma} \label{lemma:minpaths}
Let $\mathbb P_1,\mathbb P_2,\dots\mathbb P_k$ be minimal paths of the same height $h$. There exists a minimal path $\mathbb Q$ of height $h$ such that for every $i\in[k]$ there exists an onto homomorphism $\mathbb Q\to\mathbb P_i$.
\end{lemma}

We are now ready to define \emph{special} oriented trees.

\begin{definition}
Let $\mathbb T=(T;E)$ be an oriented tree of height 1. A \emph{$\mathbb T$-special tree} of height $h$ is an oriented tree $\mathbb H$ obtained from $\mathbb T$ by replacing every edge $(a,b)\in E$ with some minimal path $\mathbb P_{(a,b)}$ of height $h$, preserving orientation. (That is, identifying the initial vertex of $\mathbb P_{(a,b)}$ with  $a$ and the terminal vertex with $b$. We require the vertex sets of the minimal paths to be pairwise disjoint and also disjoint with $T$.) We will sometimes refer to $\mathbb T$ as the \emph{underlying tree structure} of the $\mathbb T$-special tree $\mathbb H$.

Throughout the paper we will denote the bottom and top levels of $\mathbb H$ by $A$ and $B$, respectively (that is, we have that $T=A\mathbin{\dot{\cup}}B$ and $E\subseteq A\times B$). In figures we mark vertices from $A$ and $B$ by \begin{tikzpicture}
[thick,
blacksquare/.style={rectangle,draw,outer sep=2pt,inner sep=0pt,minimum size=1.8mm,fill}]
\node [blacksquare] {};
\end{tikzpicture}
and
\begin{tikzpicture}
[thick,
whitesquare/.style={rectangle,draw,outer sep=2pt,inner sep=0pt,minimum size=1.8mm}]
\node [whitesquare] {};
\end{tikzpicture}; for other vertices of $\mathbb H$ we use the symbol $\bullet$.

We will use the notation ``$(a,b)\in E$'' to denote edges in the underlying tree structure $\mathbb T$ (which correspond to the minimal paths $\mathbb P_{(a,b)}$ in $\mathbb H$). For edges of the $\mathbb T$-special tree we write ``$x\myrightarrow y$ in $\mathbb H$''. In figures we use
\begin{tikzpicture}
[thick, node distance = 0.8cm,
dot/.style={circle,draw,outer sep=2pt,inner sep=0pt,minimum size=1.5mm,fill},
decoration={markings,mark=at position 1 with {\arrow[scale=1.2,black]{latex}};}]
\node [dot] (0) at (0,0) {};
\node [dot,right =of 0] (1) {};
\draw[postaction={decorate}] (0) -- (1);
\end{tikzpicture}
to depict edges of $\mathbb H$ and
\begin{tikzpicture}
[thick, node distance = 0.8cm,
blacksquare/.style={rectangle,draw,outer sep=2pt,inner sep=0pt,minimum size=1.8mm,fill},whitesquare/.style={rectangle,draw,outer sep=2pt,inner sep=0pt,minimum size=1.8mm},>=latex]
\node [blacksquare] (0) at (0,0) {};
\node [whitesquare,right =of 0] (1)  {};
\draw[->] [decorate,decoration={snake,segment length=1.5mm,amplitude=0.3mm,post length = 2mm}] (0) -- (1);
\end{tikzpicture}
for edges of the underlying tree structure.

\begin{itemize}
\item A \emph{special triad} (as defined in \cite{barto_csp_2009}) is a $\mathbb T$-special tree with underlying tree structure

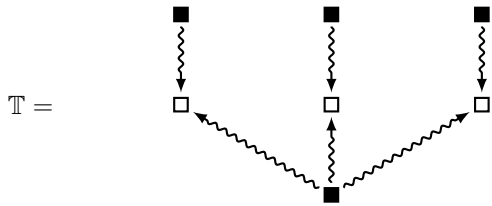
\begin{tikzpicture}
[thick, on grid, node distance=1.2cm and 2cm,
blacksquare/.style={rectangle,draw,outer sep=2pt,inner sep=0pt,minimum size=1.8mm,fill},whitesquare/.style={rectangle,draw,outer sep=2pt,inner sep=0pt,minimum size=1.8mm},>=latex]
\node [blacksquare] (0) at (0,0) {};
\node [whitesquare,above left=of 0] (1)  {};
\node [whitesquare,above =of 0] (2)  {};
\node [whitesquare,above right=of 0] (3)  {};
\node [blacksquare,above =of 1] (4)  {};
\node [blacksquare,above =of 2] (5)  {};
\node [blacksquare,above =of 3] (6)  {};
\draw[->] [decorate,decoration={snake,segment length=1.5mm,amplitude=0.3mm,post length = 2mm}] (0) -- (1);
\draw[->] [decorate,decoration={snake,segment length=1.5mm,amplitude=0.3mm,post length = 2mm}] (0) -- (2);
\draw[->] [decorate,decoration={snake,segment length=1.5mm,amplitude=0.3mm,post length = 2mm}] (0) -- (3);
\draw[->] [decorate,decoration={snake,segment length=1.5mm,amplitude=0.3mm,post length = 2mm}] (4) -- (1);
\draw[->] [decorate,decoration={snake,segment length=1.5mm,amplitude=0.3mm,post length = 2mm}] (5) -- (2);
\draw[->] [decorate,decoration={snake,segment length=1.5mm,amplitude=0.3mm,post length = 2mm}] (6) -- (3);
\node [left=of 1] (T) {$\mathbb T=$};
\end{tikzpicture}

\item A \emph{special polyad} (as defined in \cite{barto_csp_2013}) is a $\mathbb T$-special tree with underlying tree structure

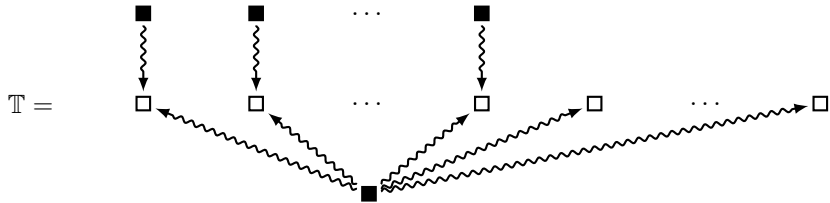
\begin{tikzpicture}
  [thick, on grid, node distance=1.2cm and 1.5cm,
    blacksquare/.style={rectangle,draw,outer sep=2pt,inner sep=0pt,minimum size=1.8mm,fill},
    whitesquare/.style={rectangle,draw,outer sep=2pt,inner sep=0pt,minimum size=1.8mm},>=latex]
\node [blacksquare] (0) at (0,0) {};
\node [above=of 0] (a) {\dots};
\node [above=of a] {\dots};
\node [whitesquare,left=of a] (2)  {};
\node [blacksquare,above=of 2] (7)  {};
\node [whitesquare,left=of 2] (1)  {};
\node [blacksquare,above=of 1] (6) {};
\node [whitesquare,right=of a] (3) {};
\node [blacksquare,above=of 3] (8)  {};
\node [whitesquare,right=of 3] (4) {};
\node [right=of 4] (b) {\dots};
\node [whitesquare,right=of b] (5)  {};
\draw[->] [decorate,decoration={snake,segment length=1.5mm,amplitude=0.3mm,post length = 2mm}] (0) -- (1);
\draw[->] [decorate,decoration={snake,segment length=1.5mm,amplitude=0.3mm,post length = 2mm}] (0) -- (2);
\draw[->] [decorate,decoration={snake,segment length=1.5mm,amplitude=0.3mm,post length = 2mm}] (0) -- (3);
\draw[->] [decorate,decoration={snake,segment length=1.5mm,amplitude=0.3mm,post length = 2mm}] (0) -- (4);
\draw[->] [decorate,decoration={snake,segment length=1.5mm,amplitude=0.3mm,post length = 2mm}] (0) -- (5);
\draw[->] [decorate,decoration={snake,segment length=1.5mm,amplitude=0.3mm,post length = 2mm}] (6) -- (1);
\draw[->] [decorate,decoration={snake,segment length=1.5mm,amplitude=0.3mm,post length = 2mm}] (7) -- (2);
\draw[->] [decorate,decoration={snake,segment length=1.5mm,amplitude=0.3mm,post length = 2mm}] (8) -- (3);
\node [left=of 1] {$\mathbb T=$};
\end{tikzpicture}

\item A \emph{special tree} is simply a $\mathbb T$-special tree for an arbitrary height 1 oriented tree $\mathbb T$.
\end{itemize}
\end{definition}

As an example, in Figure \ref{figure:NPc_triad} below we present a special triad constructed in \cite{barto_csp_2009} that gives rise to an \textbf{NP}-complete $\mathbb H$-coloring problem (and is conjectured to be the smallest oriented tree with this property).

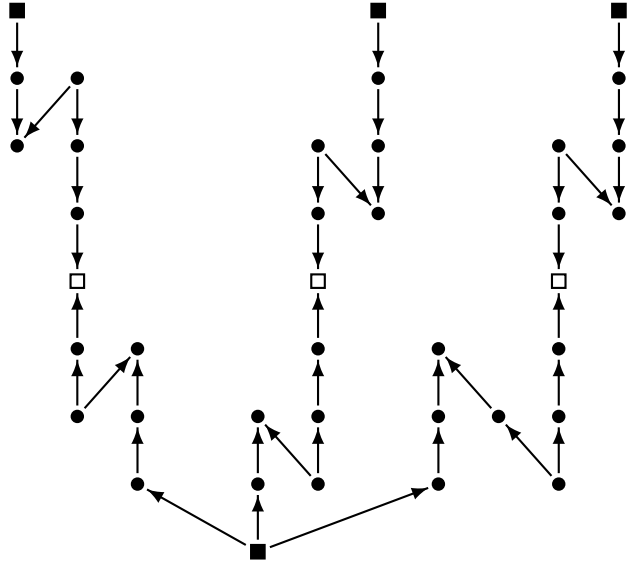
\begin{figure}[ht!]
\centering
\begin{tikzpicture}
[thick,on grid,xscale=0.8,yscale=0.9,
dot/.style={circle,draw,outer sep=2pt,inner sep=0pt,minimum size=1.5mm,fill},blacksquare/.style={rectangle,draw,outer sep=2pt,inner sep=0pt,minimum size=1.8mm,fill},whitesquare/.style={rectangle,draw,outer sep=2pt,inner sep=0pt,minimum size=1.8mm},
decoration={markings,mark=at position 1 with {\arrow[scale=1.2,black]{latex}};}]
\node [blacksquare] (0) at (0,0) {};
\node [dot] (1) at (-2,1) {};
\node [dot] (2) at (-2,2) {};
\node [dot] (3) at (-2,3) {};
\node [dot] (4) at (-3,2) {};
\node [dot] (5) at (-3,3) {};
\node [whitesquare] (6) at (-3,4) {};
\node [dot] (7) at (-3,5) {};
\node [dot] (8) at (-3,6) {};
\node [dot] (9) at (-3,7) {};
\node [dot] (10) at (-4,6) {};
\node [dot] (11) at (-4,7) {};
\node [blacksquare] (12) at (-4,8) {};
\draw[postaction={decorate}] (0) -- (1);
\draw[postaction={decorate}] (1) -- (2);
\draw[postaction={decorate}] (2) -- (3);
\draw[postaction={decorate}] (4) -- (3);
\draw[postaction={decorate}] (4) -- (5);
\draw[postaction={decorate}] (5) -- (6);
\draw[postaction={decorate}] (7) -- (6);
\draw[postaction={decorate}] (8) -- (7);
\draw[postaction={decorate}] (9) -- (8);
\draw[postaction={decorate}] (9) -- (10);
\draw[postaction={decorate}] (11) -- (10);
\draw[postaction={decorate}] (12) -- (11);

\node [dot] (a1) at (0,1) {};
\node [dot] (a2) at (0,2) {};
\node [dot] (a3) at (1,1) {};
\node [dot] (a4) at (1,2) {};
\node [dot] (a5) at (1,3) {};
\node [whitesquare] (a6) at (1,4) {};
\node [dot] (a7) at (1,5) {};
\node [dot] (a8) at (1,6) {};
\node [dot] (a9) at (2,5) {};
\node [dot] (a10) at (2,6) {};
\node [dot] (a11) at (2,7) {};
\node [blacksquare] (a12) at (2,8) {};
\draw[postaction={decorate}] (0) -- (a1);
\draw[postaction={decorate}] (a1) -- (a2);
\draw[postaction={decorate}] (a3) -- (a2);
\draw[postaction={decorate}] (a3) -- (a4);
\draw[postaction={decorate}] (a4) -- (a5);
\draw[postaction={decorate}] (a5) -- (a6);
\draw[postaction={decorate}] (a7) -- (a6);
\draw[postaction={decorate}] (a8) -- (a7);
\draw[postaction={decorate}] (a8) -- (a9);
\draw[postaction={decorate}] (a10) -- (a9);
\draw[postaction={decorate}] (a11) -- (a10);
\draw[postaction={decorate}] (a12) -- (a11);

\node [dot] (b1) at (3,1) {};
\node [dot] (b2) at (3,2) {};
\node [dot] (b3) at (3,3) {};
\node [dot] (b4) at (4,2) {};
\node [dot] (b5) at (5,1) {};
\node [dot] (b6) at (5,2) {};
\node [dot] (b7) at (5,3) {};
\node [whitesquare] (b8) at (5,4) {};
\node [dot] (aa7) at (5,5) {};
\node [dot] (aa8) at (5,6) {};
\node [dot] (aa9) at (6,5) {};
\node [dot] (aa10) at (6,6) {};
\node [dot] (aa11) at (6,7) {};
\node [blacksquare] (aa12) at (6,8) {};
\draw[postaction={decorate}] (0) -- (b1);
\draw[postaction={decorate}] (b1) -- (b2);
\draw[postaction={decorate}] (b2) -- (b3);
\draw[postaction={decorate}] (b4) -- (b3);
\draw[postaction={decorate}] (b5) -- (b4);
\draw[postaction={decorate}] (b5) -- (b6);
\draw[postaction={decorate}] (b6) -- (b7);
\draw[postaction={decorate}] (b7) -- (b8);
\draw[postaction={decorate}] (aa7) -- (b8);
\draw[postaction={decorate}] (aa8) -- (aa7);
\draw[postaction={decorate}] (aa8) -- (aa9);
\draw[postaction={decorate}] (aa10) -- (aa9);
\draw[postaction={decorate}] (aa11) -- (aa10);
\draw[postaction={decorate}] (aa12) -- (aa11);
\end{tikzpicture}
\caption{A special triad; the smallest known oriented tree with \textbf{NP}-complete $\mathbb H$-coloring problem (39 vertices).}
\label{figure:NPc_triad}
\end{figure}

The following theorem is the main algebraic result of our paper.

\begin{theorem} \label{theorem:main}
Let $\mathbb H$ be a special tree. If the algebra of idempotent polymorphisms of $\mathbb H$ is Taylor, then it is congruence meet-semidistributive.
\end{theorem}

As a consequence, we confirm the dichotomy of $\mathbb H$-coloring for special trees.

\begin{corollary} \label{corollary:dichotomy}
The CSP dichotomy conjecture holds for special trees. For any core special tree $\mathbb H$, $\mathrm{CSP}(\mathbb H)$ is \textbf{NP}-complete or $\mathbb H$ has bounded width.
\end{corollary}

We will prove Theorem \ref{theorem:main} and Corollary \ref{corollary:dichotomy} in Section \ref{section:proof}.

\section{Algebraic tools} \label{section:algebraic_tools}

In this section, we introduce the universal algebraic tools we will use in our proof.
Recall that for a fixed relational structure $\mathbb A$, the \emph{Constraint satisfaction problem} for $\mathbb A$ is the membership problem for the set $\mathrm{CSP}(\mathbb A)=\{\mathbb X\mid \mathbb X\to\mathbb A\}$. Note that if $\mathbb A'$ is the core of $\mathbb A$, then $\mathrm{CSP}(\mathbb A)=\mathrm{CSP}(\mathbb A')$.

Of particular importance to the CSP are the following two well-known classes of finite algebras: \emph{Taylor} algebras (called ``active'' in \cite{bergman_universal_2011}) and \emph{congruence meet-semidistributive} ($\mathit{SD(\wedge)}$) algebras\footnote{Taylor and  $\mathrm{SD(\wedge)}$ algebras are also commonly referred to as ``omitting type 1'' and ``omitting types 1, 2''; this terminology comes from Tame Congruence Theory (see \cite[Chapter 8]{bergman_universal_2011}).}. Instead of providing direct definitions, we present the following characterization from \cite{maroti_existence_2008}.
\begin{definition}
A \emph{weak near-unanimity} (\emph{WNU}) operation on a set $A$ is an $n$-ary ($n\geq 2$) idempotent operation $\omega$ such that for all $x,y\in A$,
$$
\omega(x,\dots,x,y)=\omega(x,\dots,x,y,x)=\dots=\omega(y,x,\dots,x).
$$
\end{definition}

\begin{theorem}[\cite{maroti_existence_2008}] \label{theorem:WNU=Taylor_manyWNU=SDmeet}
Let $\mathbf A$ be a finite algebra.
\begin{enumerate}
\item $\mathbf A$ is Taylor, if and only if there exists a WNU operation $\omega\in\Clo(\mathbf A)$.
\item $\mathbf A$ is $\mathrm{SD}(\wedge)$, if and only if there exists $n_0$ such that for all $n\geq n_0$ there exists an $n$-ary WNU operation $\omega_n\in\Clo(\mathbf A)$.
\end{enumerate}
\end{theorem}

The Algebraic CSP dichotomy conjecture (\cite{bulatov_classifying_2005}, see also \cite[Conjecture 1]{bulatov_recent_2008}) asserts that being Taylor is what
distinguishes (algebras of idempotent polymorphisms of) tractable core relational structures from the \textbf{NP}-complete ones; the hardness part is known.

\begin{theorem}[\cite{bulatov_classifying_2005}] \label{theorem:noWNU_NPc}
Let $\mathbb A$ be a core relational structure. If $\alg\mathbb A$ is not Taylor, then $\mathrm{CSP}(\mathbb A)$ is \textbf{NP}-complete.
\end{theorem}

A relational structure $\mathbb A$ is said to have \emph{bounded width} \cite{feder_computational_1999}, if $\mathrm{CSP}(\mathbb A)$ is solvable by ``local consistency checking'' algorithm (or rather algorithmic principle). We refer the reader to \cite{barto_constraint_2014} for a detailed exposition. This property is characterized (for cores) by congruence meet-semidistributivity; the characterization was conjectured, and the ``only if'' part proved, in \cite{larose_bounded_2007}.

\begin{theorem}[\cite{barto_constraint_2014},``Bounded Width Theorem''] \label{theorem:bounded_width}
A core relational structure $\mathbb A$ has bounded width (implying that $\mathrm{CSP}(\mathbb A)$ is in \textbf{P}), if and only if $\alg\mathbb A$ is $\mathrm{SD}(\wedge)$.
\end{theorem}

The proof of the Bounded Width Theorem uncovered a new characterization of $\mathrm{SD}(\wedge)$ algebras via so-called \emph{pointing operations} as well as the concept of \emph{absorbing subuniverse}, which turned out to be quite useful even outside of the realm of congruence meet-semidistributivity (see \cite{barto_absorbing_2012,barto_maltsev_2013,barto_absorption_2017}).

\subsection{Pointing operations}
Pointing operations were first used in \cite{barto_robust_2012}. More details, as well as a proof of the characterization theorem we need, can be found in~\cite{barto_maltsev_2013}.

\begin{definition}
Let $f$ be an $n$-ary idempotent operation on a set $A$ and $X,Y$ nonempty subsets of $A$. We say that $f$ \emph{weakly points} $X$ to $Y$, if there exist $\mathbf{a^1},\dots,\mathbf{a^n}\in A^n$ such that for every $i\in[n]$ and $x\in X$ we have
$$
f(a^i_1,\dots,a^i_{i-1},x,a^i_{i+1},\dots,a^i_n)\in Y
$$
(where $x$ is in the $i$th place). We refer to $\mathbf{a^1},\dots,\mathbf{a^n}$ as \emph{witnessing tuples}.
\end{definition}
\noindent The word ``weakly'' means that we can have different witnessing tuples for different coordinates, as opposed to (strongly) pointing operations from \cite{barto_maltsev_2013}. For $f:A^k\to A$ and $g:A^n\to A$, we denote by $g\compose f$ the $kn$-ary operation on $A$ defined by
$$
(g\compose f)(x_1,\dots,x_{kn})=g(f(x_1,\dots,x_k),f(x_{k+1},\dots,x_{2k}),\dots,f(x_{(n-1)k+1},\dots,x_{nk})).
$$
We will need the following easy observation which is implicit in \cite[Proposition 2.1]{barto_maltsev_2013}.
\begin{observation} \label{observation:composing_pointing_operations}
If $f:A^k\to A$ weakly points $X$ to $Y$ and $g:A^n\to A$ weakly points $Y$ to $Z$, then $g\compose f$ weakly points $X$ to $Z$.
\end{observation}
\begin{proof}
Let the witnessing tuples for $f$ weakly pointing $X$ to $Y$ and $g$ weakly pointing $Y$ to $Z$ be $\mathbf{a^1},\dots,\mathbf{a^k}$ and $\mathbf{b^1},\dots,\mathbf{b^n}$, respectively. For $i\in[n]$ and $j\in[k]$ define $\mathbf{c^{i,j}}\in A^{nk}$ to be the following tuple:
\begin{align*}
\mathbf{c^{i,j}}=(& b^i_1,b^i_1,\dots,b^i_1,b^i_2,b^i_2,\dots,b^i_2,\dots,b^i_{i-1},b^i_{i-1},\dots,b^i_{i-1},\\
& a^j_1,a^j_2,\dots,a^j_k,b^i_{i+1},b^i_{i+1},\dots,b^i_{i+1},\dots,b^i_n,b^i_n,\dots,b^i_n),
\end{align*}
where $b^i_l$ appears $k$-times for every $l\in [n]\setminus\{i\}$. It is straightforward to verify (using idempotency of $f$) that $g\compose f$ weakly points $X$ to $Z$ with witnessing tuples $\mathbf{c^{1,1}},\mathbf{c^{1,2}},\dots,\mathbf{c^{1,k}},\mathbf{c^{2,1}},\dots,\mathbf{c^{n,k}}$.
\end{proof}

Of particular interest are term operations weakly pointing the whole algebra (or a subuniverse) to a singleton, due to the following characterization of congruence meet-semidistributivity.
\begin{definition}
Let $\mathbf A$ be a finite idempotent algebra. We say that $\mathbf A$ \emph{has a weakly pointing operation}, if there exists $\tau\in\Clo\mathbf A$ and $a\in A$ such that $\tau$ weakly points $A$ to $\{a\}$.
\end{definition}
\begin{theorem}[{\cite[Theorem 1.3]{barto_maltsev_2013}}] \label{theorem:every_subuniverse_has_pointing}
A finite idempotent algebra $\mathbf A$ is $\mathrm{SD}(\wedge)$, if and only if every nonempty subuniverse $B\leq\mathbf A$ has a weakly pointing operation.
\end{theorem}

\begin{remark*}
Using this characterization it is easy to prove that given a finite idempotent algebra $\mathbf A$, the class of all $\mathrm{SD}(\wedge)$ members of the pseudovariety generated by $\mathbf A$ (that is, quotients of subuniverses of finite powers of $\mathbf A$) is closed under taking products, subalgebras, and quotients. In particular, we will need the following fact.

\end{remark*}

\begin{lemma}[{\cite[Proposition 2.1(7)]{barto_maltsev_2013}}] \label{lemma:SDmeet_closed_under_products}
Let $\mathbf A$ be a finite idempotent algebra and $B,C$ its nonempty subuniverses. If $B$ and $C$ are $\mathrm{SD}(\wedge)$, then $B\times C$ (considered as a subuniverse of $\mathbf A^2$) is $\mathrm{SD}(\wedge)$ as well.
\end{lemma}

\subsection{Absorbing subuniverses}
We briefly introduce basic notions and facts from the theory of absorption of Barto and Kozik. For more details see \cite{barto_absorption_2017,barto_absorbing_2012,barto_maltsev_2013}.

\begin{definition}
Let $\mathbf A$ be an algebra and $B\leq\mathbf A$ a nonempty subuniverse. We say that $B$ is an \emph{absorbing subuniverse} of $\mathbf A$, and write $B\abs\mathbf A$, if there exists an idempotent $\tau\in\Clo\mathbf A$ such that
\begin{gather*}
\tau(A,B,B,\dots,B,B)\subseteq B,\\
\tau(B,A,B,\dots,B,B)\subseteq B,\\
\vdots\\
\tau(B,B,B,\dots,B,A)\subseteq B.
\end{gather*}
We call $\tau$ an \emph{absorbing operation} and say that $B$ absorbs $\mathbf A$ \emph{via} $\tau$.
\end{definition}
\noindent Note that $B$ absorbs $\mathbf A$ \emph{via} $\tau$ (say $n$-ary), if and only if $\tau$ weakly points $A$ to $B$ and any tuples $\mathbf{b^1},\dots,\mathbf{b^n}\in B^n$ can serve as witnessing tuples for that. Hence absorption is somewhat stronger than pointing operations.

In applications of absorption theory, an important role is played by algebras with no proper absorbing subuniverses, the \emph{absorption-free} algebras.

\begin{definition}
An algebra $\mathbf A$ is \emph{absorption-free}, if $|A|>1$, and $B\abs\mathbf A$ implies that $B=A$.
\end{definition}

The following corollary, which is an easy consequence of Theorem \ref{theorem:every_subuniverse_has_pointing}, will be applied several times in our proof.

\begin{corollary}[{see \cite[Corollary 2.14]{barto_maltsev_2013}}] \label{corollary:every_AF_has_pointing}
A finite idempotent algebra $\mathbf A$ is $\mathrm{SD}(\wedge)$, if and only if every absorption-free subuniverse $B\leq\mathbf A$ has a weakly pointing operation.
\end{corollary}

We will use without further notice the following easy facts about absorption:
\begin{lemma}[{\cite[Proposition 2.4]{barto_absorbing_2012}}] Let $\mathbf A$ be a finite idempotent algebra.
\begin{itemize}
\item If $B\abs\mathbf A$ and $C\abs B$, then $C\abs\mathbf A$.
\item If $B\abs\mathbf A$ (via $\tau$) and $C\leq A$ and $B\cap C\neq\emptyset$, then $B\cap C\abs C$ (via $\tau |_C$).
\end{itemize}
\end{lemma}

\section{The proof} \label{section:proof}

Let us start by introducing notation used throughout the proof. First, we fix the underlying tree structure. Let $\mathbb T=(T;E)$ be an oriented tree of height 1, with $T=A\mathbin{\dot{\cup}}B$ and $E\subseteq A\times B$. Now let $\mathbb H$ be a $\mathbb T$-special tree of height $h$ such that $\alg\mathbb H$ is Taylor. (Recall that $A$ and $B$ are the bottom and top levels of $\mathbb H$, respectively.) Our aim is to prove that $\alg\mathbb H$ is $\mathrm{SD}(\wedge)$. Below we present a high level overview of the proof.

\subsubsection*{Structure of the proof}

We divide the proof into several steps organized into subsections.
\begin{enumerate}
  \item[\emph{\ref{subsection:reduction_bottom_top}}]\emph{Reduction to the bottom and top levels.}

  We show that $A$ and $B$ are subuniverses (Lemma \ref{lemma:ABE_subuniverses}) and that it is enough to prove that these two subuniverses are  $\mathrm{SD}(\wedge)$. We use the characterization from Theorem \ref{theorem:WNU=Taylor_manyWNU=SDmeet}. The key part is Lemma \ref{lemma:extending_WNUs}: Any idempotent polymorphism of $\mathbb H$ which satisfies the WNU identities on $A$ and $B$ can be modified to obtain a WNU operation on all of $H$.

  In the rest of the proof we only use the characterization from Corollary \ref{corollary:every_AF_has_pointing}. We need to prove that every absorption-free subuniverse of $A$ or $B$ has a weakly pointing operation.

  \item[\ref{subsection:singleton_absorbing}] \emph{Singleton absorbing subuniverse.}

  In this step we prove that $A$ or $B$ must have a singleton absorbing subuniverse $\{o\}$ (Lemmata \ref{lemma:WNUabsorbing_E2o} and \ref{lemma:singleton_absorbing}). This implies that there is a lot of ``absorption'' in $A$ and $B$, and absorption-free subuniverses are ``rare''. In particular, in Corollary \ref{corollary:AF_has_single_distance} we show that in any absorption-free subuniverse, all elements must have the same distance from $o$ (measured in the underlying tree structure $\mathbb T$).

  \item[\ref{subsection:Eneighbourhoods}] \emph{$E$-neighbourhoods of singletons are $\mathrm{SD}(\wedge)$.}

  In the next step we show that if an absorption-free subuniverse $C$ lies in the neighbourhood (in the underlying tree structure) of some vertex $b$ from $A\cup B$, then it has a weakly pointing operation. First we show that ${b}$ ``absorbs'' elements which are ``farther from $\{o\}$ than $C$'' via a certain binary operation (Lemma \ref{lemma:star_absorbs}). This allows for an intricate construction to obtain a weakly pointing operation for $C$ (Lemmata \ref{lemma:binary_polymorphisms_on_C} and \ref{lemma:C_has_pointing}).

  \item[\ref{subsection:all_AF}] \emph{All absorption-free subuniverses are $\mathrm{SD}(\wedge)$.}

  Finally, we show that every absorption-free subuniverse $C$ of $A$ or $B$ has a weakly pointing operation (Lemma \ref{lemma:every_AF_of_A_has_pointing}). This step is relatively easy using the previous step and induction on the distance of $C$ from the singleton absorbing subuniverse $\{o\}$.

\end{enumerate}

In most of the proof we only reason about compatibility of various operations with the underlying tree structure. The only places where we need to talk about concrete edges of $\mathbb H$ is in Lemma \ref{lemma:extending_WNUs} and Lemma \ref{lemma:binary_polymorphisms_on_C}. Recall that we use the notation ``$(a,b)\in E$'' for edges in the underlying tree structure and ``$u\myrightarrow v$ in $\mathbb H$'' for edges in the special tree $\mathbb H$.

\subsection{Reduction to the bottom and top levels}\label{subsection:reduction_bottom_top}

Our first step is to show that we can focus only on the top and bottom levels of $\mathbb H$, i.e., the sets (indeed, subuniverses) $A$ and $B$. This is the property that justifies the definition of special trees; the key ingredient is the fact about minimal paths stated in Lemma \ref{lemma:minpaths}. The reduction was already described in \cite[Lemma 4.4]{barto_csp_2013} (for special polyads). Below we present a somewhat simpler argument.

We start by showing that $A$, $B$ and $E$ (that is, the bottom and top levels of $\mathbb H$, and the edge relation of the underlying tree structure $\mathbb T$) are preserved by idempotent polymorphisms of $\mathbb H$.

\begin{lemma} \label{lemma:ABE_subuniverses}
Both $A$ and $B$ are subuniverses of $\alg\mathbb H$. Moreover, $E$ (which is a subset of $A\times B$) is a subuniverse of $(\alg\mathbb H)^2$.
\end{lemma}
\begin{proof}
By Lemma \ref{lemma:pp-definable=subuniverse}, it is enough to show that $A$, $B$ and $E$ are primitive positive definable from $\mathbb H$ with constants (although, in fact, we will not need the constants). Let $\mathbb Q$ be a minimal oriented path of height $h$ which maps homomorphically onto $\mathbb P_e$ for all $e\in E$, provided by Lemma \ref{lemma:minpaths}. Let us denote by $u$ and $v$ the initial and terminal vertex of $\mathbb Q$, respectively. The binary relation $E$ is equal to the set
$$
\{(\varphi(u),\varphi(v))\mid\varphi:\mathbb Q\to\mathbb H\text{ is a homomorphism}\},
$$
which can be expressed by a primitive positive formula. Consequently, $A(x)=(\exists y)((x,y)\in E)$ and $B(x)=(\exists y)((y,x)\in E)$ provides us with primitive positive definitions of $A$ and $B$, respectively.
\end{proof}

It is useful to observe that an $n$-ary polymorphism can be defined on different connected components of $\mathbb H^n$ independently; to verify that it preserves the edges one has to be concerned with inputs from one component at a time only. Among the connected components of $\mathbb H^n$, the most important one is the component containing the diagonal, as we show in the next lemma.

 For $n>0$ we denote by $\Delta_n$ the connected component of the digraph $\mathbb H^n$ containing the diagonal (i.e., the set $\{(v,\dots,v):v\in H\}$).

\begin{lemma}
For any $n>0$, $(A^n\cup B^n)\subseteq\Delta_n$.
\end{lemma}
\begin{proof}
It is easily seen that the set $(A^n\cup B^n)$ is connected in the digraph $\mathbb T^n$. Let $(\mathbf a,\mathbf b)$ be an edge in $\mathbb T^n$ (i.e., $(a_i,b_i)\in E$ for $i\in[n]$). Let $\mathbb Q$ be a minimal oriented path of height $h$ which maps homomorphically onto all the paths $\{\mathbb P_{(a_i,b_i)}\mid i\in[n]\}$, whose existence is provided by Lemma \ref{lemma:minpaths}. For every $i\in[n]$
let $\varphi_i:\mathbb Q\to\mathbb P_{(a_i,b_i)}$ be a homomorphism. Then the mapping $\Phi:\mathbb Q\to\mathbb H^n$ given by $\Phi(x)=(\varphi_1(x),\dots,\varphi_n(x))$ is also a homomorphism and it maps the initial and terminal vertex of $\mathbb Q$ to $\mathbf a$ and $\mathbf b$, respectively. This shows that $\mathbf a$ and $\mathbf b$ are connected  in $\mathbb H^n$ (via $\Phi(\mathbb Q)$). Consequently, the whole set $(A^n\cup B^n)$
is connected in $\mathbb H^n$. As it intersects the diagonal, it follows that $(A^n\cup B^n)\subseteq\Delta_n$.
\end{proof}

In the next lemma, we prove that every polymorphism which is a WNU operation on the top and bottom levels can be modified to obtain a polymorphism satisfying the WNU property everywhere. In Corollary \ref{corollary:reduction_to_top_and_bottom} below we combine this fact with Theorem \ref{theorem:WNU=Taylor_manyWNU=SDmeet} to obtain the desired result. The assumption that $n>2$ is there only to avoid a technical nuisance; in fact, the claim is true for $n=2$ as well (see \cite{barto_csp_2013}).

\begin{lemma} \label{lemma:extending_WNUs}
Let $n\geq 3$ and let $\tau\in\IdPol_n(\mathbb H)$ be such that $\tau |_A$ and $\tau |_B$ are WNU operations on $A$ and $B$, respectively. Then there exists $\tau'\in\IdPol_n(\mathbb H)$ which is a WNU operation on $H$.
\end{lemma}
\begin{proof}
Let us fix an arbitrary linear order $\leq_E$ of the set $E$. We define the following linear order $\sqsubseteq$ on the set $H\setminus (A\cup B)$: for $x\in\mathbb P_{(a,b)}$ and $y\in \mathbb P_{(a',b')}$ we put $x\sqsubseteq y$ if
\begin{itemize}
\item $(a,b)<_E (a',b')$, or
\item $(a,b)=(a',b')$ and $\dist_{\mathbb P_{(a,b)}}(x,a)\leq\dist_{\mathbb P_{(a,b)}}(y,a)$.
\end{itemize}
The linear order $\sqsubseteq$ was tailored to satisfy the following claim.
\begin{claim}\label{claim:extending_WNUs_minimal_elements_edge}
  Let $x_i,y_i\in H\setminus (A\cup B)$ be such that $x_i\myrightarrow y_i$ in $\mathbb H$, for $i\in[n]$. Let $x$ and $y$ be the $\sqsubseteq$-minimal elements of $\{x_1,\dots,x_n\}$ and $\{y_1,\dots,y_n\}$, respectively. Then $x\myrightarrow y$ in $\mathbb H$.
\end{claim}
Claim \ref{claim:extending_WNUs_minimal_elements_edge} follows from the fact that the binary $\sqsubseteq$-minimum operation is a polymorphism of the subgraph of $\mathbb H$ induced on $H\setminus(A\cup B)$. We include a detailed proof.
\begin{subproof}[Proof of Claim \ref{claim:extending_WNUs_minimal_elements_edge}]

Note that for every $i\in[n]$, both $x_i$ and $y_i$ lie on the same minimal path. Therefore the same must be true for the $\sqsubseteq$-minimal elements: $x,y\in \mathbb P_{(a,b)}$, where $(a,b)$ is $\leq_E$-minimal among $e\in E$ such that
$$
\mathbb P_e\cap\{x_1,\dots,x_n,y_1,\dots,y_n\}\neq\emptyset.
$$
Let $i_1,\dots,i_k\in[n]$ be a list of all indices such that $x_{i_j},y_{i_j}\in \mathbb P_{(a,b)}$. Consider the following binary operation on $\mathbb P_{(a,b)}$, which always chooses the vertex closer to the initial vertex:
$$
u\wedge v=\begin{cases}
  u, \text{ if }\dist_{\mathbb P_{(a,b)}}(u,a)\leq\dist_{\mathbb P_{(a,b)}}(v,a),\\
  v, \text{ else}.
\end{cases}
$$
Note that $x=x_{i_1}\wedge x_{i_2}\wedge\dots\wedge x_{i_k}$ and $y=y_{i_1}\wedge y_{i_2}\wedge\dots\wedge y_{i_k}$. The existence of the edge $x\myrightarrow y$ in $\mathbb H$ follows from the easy fact that $\wedge$ is a polymorphism of $\mathbb P_{(a,b)}$ (this is true for an arbitrary oriented path, see \cite[Theorem 14]{larose_algebra_2017}).
\end{subproof}

We are now ready to define $\tau'$. We split the definition into several cases and subcases. Fix $\mathbf x\in H^n$.
\begin{enumerate}
\item If $\mathbf x\in A^n\cup B^n$, then we set $\tau'(\mathbf x)=\tau(\mathbf x)$.
\item If $\mathbf x\in\Delta_n\setminus (A^n\cup B^n)$, then
\begin{enumerate}
\item if $\{x_1,\dots,x_n\}\subseteq \mathbb P_{(a,b)}$ for some $(a,b)\in E$, then we define $\tau'(\mathbf x)$ to be the $\sqsubseteq$-minimal element from $\{x_1,\dots,x_n\}$,
\item if there exists $i\in[n]$ and $e\neq e'\in E$ such that $x_i\in \mathbb P_e$ and $x_j\in \mathbb P_{e'}$ for all $j\neq i$, then we define
$$
\tau'(\mathbf x)=\tau(x_i,x_1,\dots,x_{i-1},x_{i+1},\dots,x_n),
$$
\item in all other cases we set $\tau'(\mathbf x)=\tau(\mathbf x)$.
\end{enumerate}
\item If $\mathbf x\notin\Delta_n$, then
\begin{enumerate}
\item if $\lvl(x_1)=\lvl(x_2)=\dots=\lvl(x_n)$, then we define $\tau'(\mathbf x)$ to be the $\sqsubseteq$-minimal element from $\{x_1,\dots,x_n\}$,
\item if there exists $i\in[n]$ and $k\neq l$ such that $\lvl(x_i)=k$ and $\lvl(x_j)=l$ for all $j\neq i$, then we define $\tau'(\mathbf x)=x_i$,
\item in all other cases we define $\tau'(\mathbf x)=x_1$.
\end{enumerate}

\end{enumerate}

Let us first comment on subcase (2b) of the construction. Since $\tau$ is a polymorphism, for any $(a_i,b_i)\in E$, $i\in[n]$, it induces a homomorphism from $\Delta_n\cap\prod_{i=1}^n\mathbb P_{(a_i,b_i)}$ (as an induced subgraph of $\mathbb H^n$) to $\mathbb P_{(\tau(\mathbf a),\tau(\mathbf b))}$. However, typically there are many such homomorphisms. Even if $\tau(\mathbf a)=\tau(\mathbf{a'})$, $\tau(\mathbf b)=\tau(\mathbf{b'})$
and $\mathbf{a'},\mathbf{b'}$ are just permutations of $\mathbf a,\mathbf b$, the two corresponding homomorphisms induced by $\tau$ can be different. That is why we cannot simply define $\tau'(\mathbf x)=\tau(\mathbf x)$ in subcase (2b); the WNU property might not hold.

We divide the proof of the fact that $\tau'$ is a WNU polymorphism of $\mathbb H$ into two separate claims.

\begin{claim}\label{claim:extending_WNUs_polymorphism}
$\tau'$ is a polymorphism of $\mathbb H$.
\end{claim}
\begin{subproof}[Proof of Claim \ref{claim:extending_WNUs_polymorphism}]
Let $(\mathbf x,\mathbf y)$ be an edge in $\mathbb H^n$. For every $i\in[n]$ let $e_i=(a_i,b_i)\in E$ be such that $x_i,y_i\in \mathbb P_{e_i}$. We need to show that $\tau'(\mathbf x)\myrightarrow\tau'(\mathbf y)$ in~$\mathbb H$.

We divide the proof of this claim into separate arguments depending on which cases of the construction were used to define $\tau'(\mathbf x)$ and $\tau'(\mathbf y)$. There are three options:
\begin{enumerate}[I.]
  \item The tuple $\mathbf x$ falls under case (1) of the construction, $\mathbf y$ falls under one of the subcases of~(2). This happens when $\mathbf x=\mathbf a$ (the tuple of initial vertices of the paths $\mathbb P_{e_i}$). For every $i\in[n]$, $y_i$ must be the unique vertex from $\mathbb P_{e_i}$ such that $x_i=a_i\myrightarrow y_i$ in $\mathbb H$. We split the argument depending on the subcase of the construction applied to $\tau'(\mathbf y)$:

  \begin{enumerate}[(2a)]
    \item  In this subcase, $e_1=\dots=e_n=e$ for some $e=(a,b)\in E$ and $y_1=\dots=y_n=\tau'(\mathbf y)=y$,
      where $y$ is the unique vertex from $\mathbb P_e$ such that $a\myrightarrow y$ in $\mathbb H$. Hence $\tau'(\mathbf x)\myrightarrow\tau'(\mathbf y)$ in $\mathbb H$.

    \item We have that
    $$\mathbf x=\mathbf a=(a,\dots,a,a',a,\dots,a)$$
    for some $a,a'\in A$ where $a'$ is in the $i$th coordinate (possibly $i=1$ or $i=n$), and
    $$\mathbf y=(y,\dots,y,y',y,\dots,y).$$
    Using both that $\tau$ satisfies the WNU property on $A$ and that it is a polymorphism of $\mathbb H$, we get that
    $$\tau'(\mathbf x)=\tau(\mathbf a)=\tau(a',a,\dots,a)\myrightarrow\tau(y',y,\dots,y)=\tau'(\mathbf y)\text{ in }\mathbb H.$$

    \item In this subcase, $\tau'(\mathbf x)=\tau(\mathbf a)\myrightarrow\tau(\mathbf y)=\tau'(\mathbf y)$ in $\mathbb H$ follows from the fact that $\tau$ is a polymorphism of $\mathbb H$.
  \end{enumerate}

  \item The tuple $\mathbf y$ falls under case (1) of the construction, $\mathbf x$ falls under one of the subcases of~(2). This happens when $\mathbf y=\mathbf b$ (the tuple of terminal vertices of the paths $\mathbb P_{e_i}$). The proof is analogous to the previous option.

  \item In all other situations, both $\mathbf x$ and $\mathbf y$ must fall under the same case and subcase of the construction, one of the subcases of (2) or (3). (To see why, note that $\mathbf x\myrightarrow\mathbf y$ in $\mathbb H^n$ implies that both $\mathbf x$ and $\mathbf y$ lie in the same connected component of $\mathbb H^n$, and that the path $\mathbb P_{e_i}$ is the same for $x_i$ and $y_i$, for every $i\in[n]$.) Each subcase requires a slightly different argument:
  \begin{itemize}
    \item[(2a)] In this subcase, $\tau'(\mathbf x)$ is the $\sqsubseteq$-minimal element of $\{x_1,\dots,x_n\}$ and $\tau'(\mathbf y)$ is the $\sqsubseteq$-minimal element of $\{y_1,\dots,y_n\}$. It follows from Claim~\ref{claim:extending_WNUs_minimal_elements_edge} that $\tau'(\mathbf x)\myrightarrow\tau'(\mathbf y)$ in $\mathbb H$.

    \item[(3a)] Similarly as in subcase (2a), we use Claim~\ref{claim:extending_WNUs_minimal_elements_edge}.

    \item[(2b)] Note that the distinguished coordinate $i\in[n]$ and the paths $\mathbb P_e,\mathbb P_{e'}$ are the same for both $\mathbf x$ and $\mathbf y$. The construction says that
    \begin{align*}
      \tau'(\mathbf x)&=\tau(x_i,x_1,\dots,x_{i-1},x_{i+1},\dots,x_n),\\
      \tau'(\mathbf y)&=\tau(y_i,y_1,\dots,y_{i-1},y_{i+1},\dots,y_n).
    \end{align*}
    Therefore $\tau'(\mathbf x)\myrightarrow\tau'(\mathbf y)$ in $\mathbb H$
    follows from the fact that $\tau$ is a polymorphism of $\mathbb H$.

    \item[(3b)] Similarly as in subcase (2b), the distinguished coordinate $i\in[n]$ is the same for both $\mathbf x$ and $\mathbf y$. Therefore $\tau'(\mathbf x)=x_i$, $\tau'(\mathbf y)=y_i$, and  we know that $x_i\myrightarrow y_i$ in $\mathbb H$.

    \item[(2c)] This subcase is easy: $\tau'(\mathbf x)=\tau(\mathbf x)$, $\tau'(\mathbf y)=\tau(\mathbf y)$, and $\tau$ is a polymorphism of~$\mathbb H$.

    \item[(3c)] This subcase is also easy: $\tau'(\mathbf x)=x_1$, $\tau'(\mathbf y)=y_1$, and  $x_1\myrightarrow y_1$ in~$\mathbb H$.
  \end{itemize}
\end{enumerate}
We proved that in all possible situations, $\tau'(\mathbf x)\myrightarrow\tau'(\mathbf y)$ in $\mathbb H$. Therefore $\tau'$ is a polymorphism of $\mathbb H$.
\end{subproof}

\begin{claim}\label{claim:extending_WNUs_WNU}
$\tau'$ is a WNU operation on H.
\end{claim}
\begin{subproof}[Proof of Claim \ref{claim:extending_WNUs_WNU}]
Let $x,y\in H$ be arbitrary. Note that all of the tuples
$$(y,x,x,\dots,x),(x,y,x,\dots,x),\dots,(x,x,\dots,x,y)$$
fall under the same case and subcase of the construction, and that it can be neither (2c) nor (3c). In case (1) the WNU property follows from the fact that $\tau$ is a WNU operation on $A$ and $B$ while in cases (2a) and (3a) from the fact that the construction in these cases is independent of order and repetition of elements. In case (2b) the result is $\tau(y,x,\dots,x)$ for all the tuples in question while in case (3b) the result is always $y$.
\end{subproof}
We proved that $\tau'$ is indeed a WNU polymorphism of $\mathbb H$ which concludes the proof of the lemma.
\end{proof}

\begin{corollary} \label{corollary:reduction_to_top_and_bottom}
If both $A$ and $B$ are $\mathrm{SD}(\wedge)$, then $\alg\mathbb H$ is $\mathrm{SD}(\wedge)$.
\end{corollary}
\begin{proof}
By Lemma \ref{lemma:SDmeet_closed_under_products}, $A\times B$ (which is a subuniverse of $(\alg\mathbb H)^2$) is $\mathrm{SD}(\wedge)$ as well. Hence, by Theorem \ref{theorem:WNU=Taylor_manyWNU=SDmeet}, there exists $n_0$ such that for every $n\geq n_0$ there exists $\tau_n\in\IdPol_n(\mathbb H)$ such that $(\tau_n\times\tau_n)|_{A\times B}$ is a WNU operation on $A\times B$. This implies that the restrictions of $\tau_n$ to $A$ and $B$ are WNU operations. Using Lemma \ref{lemma:extending_WNUs} we obtain, for every $n\geq\mathrm{max}(n_0,3)$, a WNU operation $\tau'_n\in\IdPol_n(\mathbb H)$. The proof concludes by another application of Theorem~\ref{theorem:WNU=Taylor_manyWNU=SDmeet}.
\end{proof}

\subsection{Singleton absorbing subuniverse} \label{subsection:singleton_absorbing}

Our next step is to prove that either $A$ or $B$ has a singleton absorbing subuniverse, where the absorbing operation is a WNU operation. This is the one and only place where we use the assumption that $\alg\mathbb H$ is Taylor.

Since $\alg\mathbb H$ is Taylor, by Theorem \ref{theorem:WNU=Taylor_manyWNU=SDmeet} there exists a WNU operation $\omega\in\IdPol(\mathbb H)$. Let $\mathbin{\circ}:H^2\to H$ be the \emph{binary polymer} of the WNU operation $\omega$, that is,
$$
x\mathbin{\circ}y=\omega(x,x,\dots,y)=\dots=\omega(y,x,\dots,x)
$$
for $x,y\in H$. Note that $\circ\in\IdPol_2(\mathbb H)$.

We can and will assume that $\omega$ is \emph{special} in the sense of \cite[Definition 6.2]{barto_constraint_2009}, that is, satisfies $x\mathbin{\circ}(x\mathbin{\circ}y)=x\mathbin{\circ}y$. (Here the word \emph{special} is unrelated to our definition of \emph{special} trees.) This property can be enforced by an iterated composition of $\omega$ with itself (i.e., $\omega\compose\omega\compose\dots\compose\omega$, $|H|!$-times, see \cite[Lemma 6.4]{barto_constraint_2009}).

For a subset $C\subseteq A$ we define the \emph{$E$-neighbourhood} of $C$, denoted by $E_+(C)$, to be the set $\{b\in B\mid (c,b)\in E\text{ for some }c\in C\}$. Similarly, the $E$-neighbourhood of $D\subseteq B$ is the set $E_-(D)=\{a\in A\mid (a,d)\in E\text{ for some }d\in D\}$. For brevity we write $E_+(c)$, $E_-(d)$ instead of $E_+(\{c\})$, $E_-(\{d\})$. Moreover, for every $k\geq 0$, $C\subseteq A$ and $D\subseteq B$ we inductively define the sets $E_k(C)$ and $E_k(D)$ as follows:
\begin{itemize}
\item $E_0(C)=C$ and  $E_0(D)=D$,
\item $E_1(C)=E_+(C)$ and $E_1(D)=E_-(D)$, and
\item $E_k(C)=E_1(E_{k-1}(C))$ and $E_k(D)=E_1(E_{k-1}(D))$ for $k>1$.
\end{itemize}
Note that the above definition can be reformulated as follows:
$$
E_k(C)=\{x\in A\cup B\mid (\exists\,c\in C)\ \dist_\mathbb T(x,c)\leq k\ \&\ \dist_\mathbb T(x,c) \equiv k \,(\bmod\ 2)\},
$$
and similarly for $E_k(D)$. We will frequently use the following easy facts (as well as the obvious ``dual'' versions for $D\leq D'\leq B$), which are all consequences of the fact that $E\leq(\alg\mathbb H)^2$. We leave the proof to the reader.
\begin{observation} \label{observation_basic_properties_of_E_k}
If $C\leq C'\leq A$, then the following holds:
\begin{itemize}
\item $E_+(C)\leq E_+(C')\leq B$,
\item $E_k(C)\leq A$ for $k$ even and $E_k(C)\leq B$ for $k$ odd,
\item if $k\leq l$ and $l-k$ is even, then $E_k(C)\leq E_l(C)$,
\item if $C\neq 0$, then there exists $k$ such that $E_k(C)=A$ and $E_{k+1}(C)=B$, and
\item if $C\abs C'$ via some $\tau\in\IdPol(\mathbb H)$, then for every $k\geq 0$, $E_k(C)\abs E_k(C')$ as well and, moreover, the absorption is via the same operation $\tau$. \footnote{Technically, the absorbing operation is $\tau |_{C'}$ in the first case while it is $\tau |_{E_k(C')}$ in the second case, but we will neglect this formality.}
\end{itemize}
\end{observation}

We are now ready to prove that either $A$ or $B$ has a singleton absorbing subuniverse and, moreover, that this absorption is realized via the special WNU operation~$\omega$. The proof spans the next two lemmata.

\begin{lemma} \label{lemma:WNUabsorbing_E2o}
There exists $o\in A\cup B$ such that $\{o\}\abs E_2(o)$ via $\omega$.
\end{lemma}
\begin{proof}
Suppose for contradiction that no such element exists. It follows that for every $u\in A\cup B$ there exists $w\in E_2(u)$ such that $u\mathbin{\circ}w=v\neq u$. Since the WNU operation $\omega$ is special, we have that $u\mathbin{\circ}v=u\mathbin{\circ}(u\mathbin{\circ}w)=u\mathbin{\circ}w=v$. Consider the binary relation $\gg$ on $A\cup B$ defined by setting $x\gg y$ if and only if $y\in E_2(x)\setminus\{x\}$ and $x\mathbin{\circ}y=y$. We have proved that for every $u\in A\cup B$ there exists $v$ such that $u\gg v$.

Let $k$ be maximal such that there exists a sequence $\langle u_0 u_1 \dots u_k\rangle$ of elements of $A\cup B$ with the following properties:
\begin{enumerate}
\item $\dist_\mathbb T(u_0,u_i)=i$ for all $i\in[k]$, and
\item $u_i\gg u_{i+2}$ for all $0\leq i\leq k-2$.
\end{enumerate}
Note that (1) ensures that the sequence is non-repeating and thus, by finiteness of $A\cup B$, such a maximal $k$ exists. The previous paragraph shows that $k\geq 2$: just take $\langle a,b,a'\rangle$ for any $a,a'\in E_-(b)$ such that $a\gg a'$.

Let us assume that $u_k\in A$; the proof for $u_k\in B$ is analogous. Let $u_k'\in A$ and $u_{k+1}'\in B$ be such that $u_{k-1}\gg u_{k+1}'$ and $u_{k-1},u_{k+1}'\in E_+(u_k')$ (see Figure \ref{figure:extending_the_sequence}). We will prove that the sequence $\langle u_0 u_1 \dots u_{k-1} u_k' u_{k+1}'\rangle$ also satisfies properties (1) and (2); a contradiction with maximality of $k$.

\begin{figure}[ht!]\centering
\begin{tikzpicture}
  [thick, on grid, node distance=1cm and 2cm,
  blacksquare/.style={rectangle,draw,outer sep=2pt,inner sep=0pt,minimum size=1.8mm,fill},
  whitesquare/.style={rectangle,draw,outer sep=2pt,inner sep=0pt,minimum size=1.8mm},
  >=latex]
  \node [whitesquare,label=left:{$\cdots$}] (-1)  {};
  \node [blacksquare,right=of -1,label=below:{$u_{k-2}$}] (0) {};
\node [whitesquare,right=of 0,label=below:{$u_{k-1}$}] (1)  {};
\node [blacksquare,above right=of 1,label=below:{$u_k$}] (2)  {};
\node [blacksquare,below right=of 1,label=below:{$u_k'$}] (3)  {};
\node [whitesquare,right=of 3,,label=below:{$u_{k+1}'$}] (4)  {};
\draw[->,decorate,decoration={snake,segment length=1.5mm,amplitude=0.3mm,post length = 2mm}] (0) -- (-1);
\draw[->,decorate,decoration={snake,segment length=1.5mm,amplitude=0.3mm,post length = 2mm}] (0) -- (1);
\draw[->,decorate,decoration={snake,segment length=1.5mm,amplitude=0.3mm,post length = 2mm}] (2) -- (1);
\draw[->,decorate,decoration={snake,segment length=1.5mm,amplitude=0.3mm,post length = 2mm}] (3) -- (1);
\draw[->,decorate,decoration={snake,segment length=1.5mm,amplitude=0.3mm,post length = 2mm}] (3) -- (4);
\end{tikzpicture}
\caption{Extending the sequence $\langle u_0 u_1 \dots u_k\rangle$.}\label{figure:extending_the_sequence}
\end{figure}
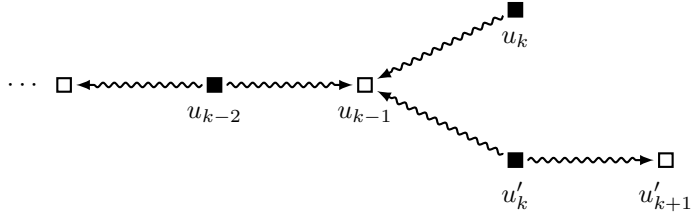

First we prove (1). From $u_{k-1}\gg u_{k+1}'$ we get that $\dist_\mathbb T(u_{k-1},u_k')=1$ and $\dist_\mathbb T(u_{k-1},u_{k+1}')=2$. Since $\mathbb T$ is a tree, it suffices to rule out the possibility that $u_k'=u_{k-2}$. In that case $(u_{k-2},u_{k+1}')\in E$, $(u_{k-2},u_{k-1})\in E$ and $(u_k,u_{k-1})\in E$ would give
$$
(\omega(u_{k-2},u_{k-2},\dots,u_{k-2},u_k),\omega(u_{k+1}',u_{k-1},\dots,u_{k-1},u_{k-1}))\in E.
$$
The left hand side is $u_{k-2}\mathbin{\circ}u_k=u_k$ while the right hand side is $u_{k-1}\mathbin{\circ}u_{k+1}'=u_{k+1}'$; and so we get $(u_k,u_{k+1}')\in E$.
But $u_k\in E_-(u_{k-1})\cap E_-(u_{k+1}')$ would imply that $u_k=u_k'=u_{k-2}$ which contradicts $u_{k-2}\gg u_k$.

To prove (2) we only need to establish $u_{k-2}\gg u_k'$. From $(u_{k-2},u_{k-1})\in E$, $(u_k',u_{k+1}')\in E$ and the fact that $\circ$ preserves $E$ we get
$$
(u_{k-2}\mathbin{\circ}u_k', \underbrace{u_{k-1}\mathbin{\circ}u_{k+1}'}_{=u_{k+1}'})\in E.
$$
On the other hand, $\{u_{k-2},u_k'\}\subseteq E_-(u_{k-1})$, which is a subuniverse, and thus $(u_{k-2}\mathbin{\circ}u_k',u_{k-1})\in E$. It follows that $u_{k-2}\mathbin{\circ}u_k'=u_k'$; and $u_{k-2}\neq u_k'$ is proved above.
\end{proof}

Fix $o\in A\cup B$ given by the previous lemma. To simplify the exposition we choose that $o\in A$. The proofs are essentially the same in the other case. (Moreover, note that reversing edges of $\mathbb H$ does not change $\alg\mathbb H$.)

Since $\mathbb H$ is an oriented tree, it follows that for every $x,y\in H$ there exists a unique oriented path connecting $x$ to $y$ in $\mathbb H$. We denote this path by $\Path_\mathbb H(x,y)$. For the purpose of the proof of Lemma \ref{lemma:singleton_absorbing} below as well as Lemmata \ref{lemma:star_absorbs} and \ref{lemma:binary_polymorphisms_on_C} from Subsection \ref{subsection:Eneighbourhoods}, we define a partial order $\preceq$ on $H$ by setting $u\preceq v$ if and only if $u\in\Path_\mathbb H(o,v)$.

Note that $o$ is the least element in this order. Furthermore, for $u,v\in A\cup B$, $u\preceq v$ implies $\dist_\mathbb T(o,u)\leq\dist_\mathbb T(o,v)$. We will also write $u\prec v$ to mean $u\preceq v$ and $u\neq v$. (If we forget the orientation of edges and consider $\mathbb H$ as a rooted tree with root $o$, then $u \prec v$ means that $v$ is a \emph{descendant} of $u$.)

\begin{lemma} \label{lemma:singleton_absorbing}
If $a,a'\in A$ and $a\preceq a'$, then $a \mathbin{\circ} a'=a$ (and similarly for $b,b'\in B$). In particular, $\{o\}\abs A$ via $\omega$.
\end{lemma}
\begin{proof}
If $a=a'$, then $a\mathbin{\circ}a'=a$ follows trivially from idempotency of $\omega$. Else, let $k\geq 0$ be such that $\dist_\mathbb T(o,a')=k+2$. We have that $a'\in E_{k+2}(o)\setminus E_k(o)$. From $a\prec a'$ we know that the distance betwen $a$ and $o$ is at most $k$ and thus $a\in E_k(o)$. From Lemma \ref{lemma:WNUabsorbing_E2o} and the last item of Observation` \ref{observation_basic_properties_of_E_k} it follows that $E_k(o)\unlhd E_{k+2}(o)$ via $\omega$ and so $a\mathbin{\circ}a'\in E_k(o)$. In particular, $a\mathbin{\circ}a'\neq a'$.

Note that $l=\dist_\mathbb T(a,a')$ is even and that there exists a unique vertex $u\in\Path_\mathbb H(o,a')\cap(A\cup B)$ such that $\dist_\mathbb T(a,u)=\dist_\mathbb T(u,a')=l/2$. Since $a,a'\in E_{l/2}(u)$, which is a subuniverse, we have that $a\mathbin{\circ}a'\in E_{l/2}(u)$ (see Figure~\ref{figure:a_circ_aprime}).

We will show that $a$ is the $\preceq$-least element of $E_{l/2}(u)$. Choose any $v\in E_{l/2}(u)$, $v\neq a$. Consider the concatenation of the paths $\Path_\mathbb H(o,v)$, $\Path_\mathbb H(v,u)$, and $\Path_\mathbb H(u,o)$. Because $a\in\Path_\mathbb H(u,o)$ and there are no cycles in $\mathbb H$, we must pass through $a$ once more. But since  $\dist_\mathbb T(a,u)=l/2\geq\dist_\mathbb T(v,u)$,
we know that $a\notin\Path_\mathbb H(v,u)$ and so it must be the case that $a\in\Path_\mathbb H(o,v)$. This means that $a\preceq v$. In particular, $a\preceq a\mathbin{\circ}a'$.

\begin{figure}[ht!]\centering
\begin{tikzpicture}
  [thick,yscale=1.2,
  blacksquare/.style={rectangle,draw,outer sep=2pt,inner sep=0pt,minimum size=1.8mm,fill},
  whitesquare/.style={rectangle,draw,outer sep=2pt,inner sep=0pt,minimum size=1.8mm,fill=white},
  >=latex]
  \draw [fill=gray!15!white] (6,0) ellipse (3.2 and 1.7);
  \draw[rounded corners=20pt,ultra thick] (6.5,-2.5) -- (7.5,-2.5) -- (7.5,2.5) -- (6.5, 2.5) ;
  \draw[rounded corners=20pt,ultra thick] (8.5,-2.5) -- (9.5,-2.5) -- (9.5,2.5) -- (8.5, 2.5) ;
  \node (label_El2) at (6, -1.4) {$E_{l/2}(u)$};
  \node (label_Eko) at (6.7, -2.1) {$E_k(o)$};
  \node (label_Eko2) at (8.7, -2.1) {$E_{k+2}(o)$};
\node [blacksquare,label=west:{$o$}] (o) at (0,0) {};
\node [label=east:{$\cdots$}](o1) at (1,0) {};
\node [whitesquare] (o+) at (0.5,1.5) {};
\node (o2) at (2,0) {};
\node [blacksquare] (a) at (3,0) {};
\node (labela) at (3.3,-0.3) {$a$};
\node [whitesquare] (a+) at (3.3,1.5) {};
\node [whitesquare] (a1) at (4,0) {};
\node [blacksquare] (a1+) at (4.25,1) {};
\node [whitesquare] (a1++) at (5,2) {};
\node [blacksquare] (a2) at (5,0) {};
\node [whitesquare] (a3+) at (5.8,0.9) {};
\node [blacksquare] (a4+) at (6.8,1.4) {};
\node [whitesquare] (a5+) at (8,1.8) {};
\node [whitesquare] (u) at (6,0) {};
\node (labelu) at (6,-0.35) {$u$};
\node [blacksquare] (a4) at (7,0) {};
\node [whitesquare] (a5) at (8,0) {};
\node [blacksquare] (a') at (9,0) {};
\node (labela') at (8.75,-0.3) {$a'$};
\node [whitesquare] (a'+) at (10,1) {};
\node at (3.2,0.2) {$\mathbf{?}$};
\node at (4.05,1.1) {$\mathbf{?}$};
\node at (4.8,0.2) {$\mathbf{?}$};
\node at (6.8,0.2) {$\mathbf{?}$};
\node at (6.6,1.5) {$\mathbf{?}$};
\draw[->,decorate,decoration={snake,segment length=1.5mm,amplitude=0.3mm,post length = 2mm}] (o) -- (o+);
\draw[->,decorate,decoration={snake,segment length=1.5mm,amplitude=0.3mm,post length = 2mm}] (o) -- (o1);
\draw[->,decorate,decoration={snake,segment length=1.5mm,amplitude=0.3mm,post length = 2mm}] (a) -- (o2);
\draw[->,decorate,decoration={snake,segment length=1.5mm,amplitude=0.3mm,post length = 2mm}] (a) -- (a+);
\draw[->,decorate,decoration={snake,segment length=1.5mm,amplitude=0.3mm,post length = 2mm}] (a) -- (a1);
\draw[->,decorate,decoration={snake,segment length=1.5mm,amplitude=0.3mm,post length = 2mm}] (a2) -- (a1);
\draw[->,decorate,decoration={snake,segment length=1.5mm,amplitude=0.3mm,post length = 2mm}] (a2) -- (u);
\draw[->,decorate,decoration={snake,segment length=1.5mm,amplitude=0.3mm,post length = 2mm}] (a4) -- (u);
\draw[->,decorate,decoration={snake,segment length=1.5mm,amplitude=0.3mm,post length = 2mm}] (a4) -- (a5);
\draw[->,decorate,decoration={snake,segment length=1.5mm,amplitude=0.3mm,post length = 2mm}] (a') -- (a5);
\draw[->,decorate,decoration={snake,segment length=1.5mm,amplitude=0.3mm,post length = 2mm}] (a') -- (a'+);
\draw[->,decorate,decoration={snake,segment length=1.5mm,amplitude=0.3mm,post length = 2mm}] (a1+) -- (a1);
\draw[->,decorate,decoration={snake,segment length=1.5mm,amplitude=0.3mm,post length = 2mm}] (a1+) -- (a1++);
\draw[->,decorate,decoration={snake,segment length=1.5mm,amplitude=0.3mm,post length = 2mm}] (a2) -- (a3+);
\draw[->,decorate,decoration={snake,segment length=1.5mm,amplitude=0.3mm,post length = 2mm}] (a4+) -- (a3+);
\draw[->,decorate,decoration={snake,segment length=1.5mm,amplitude=0.3mm,post length = 2mm}] (a4+) -- (a5+);
\end{tikzpicture}
\caption{An example of (a subgraph of) $\mathbb T$ with $l=\dist_\mathbb T(a,a')=6$. The first step shows that $a\mathbin{\circ}a'$ must be one of the vertices marked `?'. Then we repeat the argument with $a\mathbin{\circ}a'$ in the role of $a'$.}\label{figure:a_circ_aprime}
\end{figure}
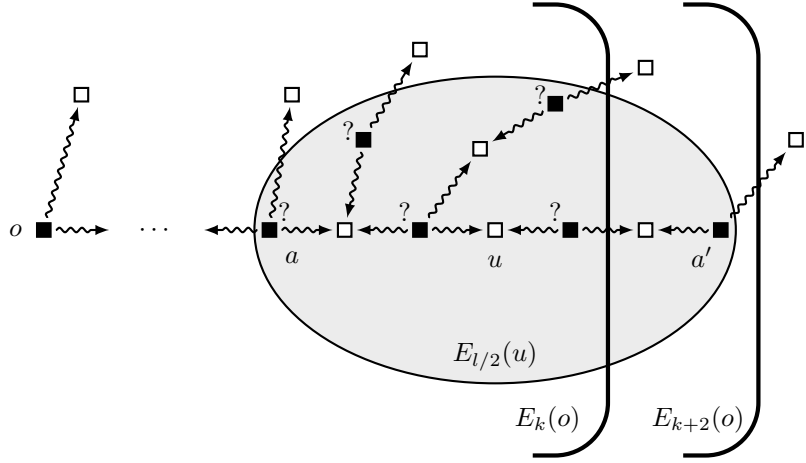

Suppose for contradiction that $a\neq a\mathbin{\circ}a'$. Then repeating the arguments from the first paragraph with $a\mathbin{\circ}a'$ in the role of $a'$ yields $a\mathbin{\circ}(a\mathbin{\circ}a')\neq a\mathbin{\circ}a'$ which contradicts the fact that $\omega$ is a special WNU operation.

Hence we have proved that $a\mathbin{\circ}a'=a$. The proof for $b\preceq b'$ is essentially the same. The fact that $\{o\}\abs A$ via $\omega$ now follows immediately from the definition of absorption and the fact that $o$ is the $\preceq$-least element of $A$.
\end{proof}

The existence of a singleton absorbing subuniverse $\{o\}$ already significantly restricts living space for possible absorption-free subuniverses in $A$ and $B$, as we can see in the following corollary. (Of course, the dual version for $D\leq B$ is also true.)

\begin{corollary} \label{corollary:AF_has_single_distance}
If $C\leq A$ is absorption-free, then there exists $k>0$ such that $\dist_\mathbb T(o,c)=k$ for all $c\in C$.
\end{corollary}
\begin{proof}
Let $k$ be the minimum from the set $\{\dist_\mathbb T(o,c)\mid c\in C\}$. Since $\{o\}\abs A$, by Observation \ref{observation_basic_properties_of_E_k} we have $E_k(o)\abs E_k(A)=A$, and thus also $C\cap E_k(o)\abs C\cap A=C$. Since $C$ is absorption-free, it follows that $C\cap E_k(o)=C$.

We have proved that $k\leq\dist_\mathbb T(o,c)\leq k$ for all $c\in C$. Note that $k>0$, since otherwise $C=\{o\}$ which is not absorption-free by definition.
\end{proof}

\subsection{$E$-neighbourhoods of singletons are $\mathrm{SD}(\wedge)$}\label{subsection:Eneighbourhoods}

In this subsection we prove that $E$-neighbourhoods of elements from $A\cup B$ are $\mathrm{SD}(\wedge)$. Our strategy is to show that whenever they have an absorption-free subuniverse, it must have a weakly pointing operation (and then apply Corollary \ref{corollary:every_AF_has_pointing}). For the rest of this subsection, we fix $b\in B$ and an absorption-free subuniverse $C\leq E_-(b)$. (The proof for $D\leq E_+(a)$ is analogous.)

We will need the following observation:

\begin{observation} \label{observation:b_prec_c}
  For all $c\in C$, $b\prec c$.
\end{observation}
\begin{proof}
Let $\dist_\mathbb T(b,o)=k$. Since $\mathbb T$ is a tree, exactly one vertex $u\in E_-(b)$ has $\dist_\mathbb T(u,o)=k-1$; all other vertices from $E_-(b)$ have $\mathbb T$-distance $k+1$ from $o$ (see Figure \ref{figure:AFsetC}). The rest follows from Corollary \ref{corollary:AF_has_single_distance} and the fact that $|C|>1$.
\end{proof}

In the first step, we prove that elements from $B$ which are $\preceq$-above $C$ are ``absorbed by $b$'' via a certain binary operation $\mathbin{\star}$. (Note that such elements do not need to form a subuniverse, and so it is not absorption in the sense we defined.) Later we will use this operation to construct various binary polymorphisms and then build up a weakly pointing operation for $C$ from them.

Let us denote by $\mathbin{\star}$ the binary idempotent polymorphism of $\mathbb H$ given by
$$
x\star y=(\dots((x\underbrace{\mathbin{\circ}\,y)\mathbin{\circ}y)\mathbin{\circ}\dots\mathbin{\circ}}_{|H|\text{-times}}y),
$$
where the operation $\mathbin{\circ}$ appears $|H|$-times (just for good measure).

\begin{lemma} \label{lemma:star_absorbs}
If $d\in B$ is such that $c\prec d$ for some $c\in C$, then $b\mathbin{\star} d=d\mathbin{\star} b=b$.
\end{lemma}
\begin{proof}
Since $b\prec c$ (by Observation \ref{observation:b_prec_c}) and $c\prec d$, we also have that $b\prec d$. From Lemma \ref{lemma:singleton_absorbing} (applied to $b,d\in B$ in the role of $a,a'\in A$)  we get that $b\mathbin{\circ}d=b$, and thus also $b \mathbin{\star} d=b$. To prove the other equality, consider the sequence $\langle d_0 d_1 \dots d_{|H|}\rangle$ of elements of $B$ defined inductively by setting $d_0=d$ and $d_i=d_{i-1}\mathbin{\circ}b$ for $i\in[|H|]$. Observe that $d_{|H|}=d\mathbin{\star}b$ which we want to equate to $b$.

Let $k_i$ denote the distance $\dist_\mathbb T(d_i,b)$. We will prove that for every $0\leq i\leq |H|$, $b\preceq d_i$ and $k_i\leq k_{i-1}$ (we set $k_{-1}=k_0$). The proof uses induction on $i$; the case $i=0$ is trivial. Assume that the claim holds for some $i<|H|$. Following the same logic as in the proof of Lemma \ref{lemma:singleton_absorbing}, there exists $u_i\in A\cup B$
such that $\dist_\mathbb T(u_i,b)=\dist_\mathbb T(u_i,d_i)={k_i}/2$. Since $b\preceq d_i$, it follows that $b$ is the $\preceq$-minimal (and $d_i$ a $\preceq$-maximal) element of $E_{{k_i}/2}(u_i)$. Consequently, $d_{i+1}=d_i\mathbin{\circ} b\in E_{{k_i}/2}(u_i)$ implies that $b\preceq d_{i+1}$ and $k_{i+1}\leq k_i$
(see Figure \ref{figure:AFsetC}).

\begin{figure}[ht!]
\centering
\begin{tikzpicture}
  [thick, on grid, node distance=0.8cm and 1.2cm,
  blacksquare/.style={rectangle,draw,outer sep=2pt,inner sep=0pt,minimum size=1.8mm,fill},
  whitesquare/.style={rectangle,draw,outer sep=2pt,inner sep=0pt,minimum size=1.8mm},
  >=latex]
\node [blacksquare,label=east:{$o$},label=north:{$\vdots$}] (00) at (0,0) {};
\node [above=of 00] (01) {};
\node [whitesquare,above=of 01,label=south west:{$b$}] (02) {};
\node [above=of 02] (03) {};
\node [right=of 03] (13) {};
\node [right=of 13] (23) {};
\node [left=of 03] (-13) {};
\node [left=of -13] (-23) {};
\node [left=of -23] (-33) {};
\node [left=of -33] (-43) {};
\node [above=of -43] (-44) {};
\node [above=of -44] (-45) {};
\draw[rounded corners=22pt,fill=gray!20!white] (-45) rectangle (23) {};
\node [below=of -45,label=west:{$C$}] {};
\node [above=of 03] (04) {$\cdots$};
\node [blacksquare,right=of 04,label=south east:{$c'$}] (14) {};
\node [right=of 14] (24) {};
\node [above=of 14] (15) {};
\node [above=of 15] (16) {};
\node [whitesquare,left=of 16,label=north:{$\vdots$}] (06) {};
\node [whitesquare,right=of 16,label=north:{$\vdots$}] (26) {};
\node [blacksquare,right=of 23,label=north west:{$\ddots$}] (33) {};
\node [blacksquare,left=of 04] (-14) {};
\node [left=of -14] (-24) {};
\node [blacksquare,left=of -24,label=south west:{$c$}] (-34) {};
\node [above=of -34] (-35) {};
\node [above=of -35] (-36) {};
\node [left=of -35] (-45) {};
\node [whitesquare,left=of -36,label=north west:{$u_i$}] (-46) {};
\node [whitesquare,right=of -36,label=north east:{$d_{i+1}$}] (-26) {};
\node [above=of -46] (-47) {};
\node [blacksquare,above=of -47] (-48) {};
\node [above=of -48] (-49) {};
\node [whitesquare,above=of -49,label=north west:{$d_i$}] (-410) {};
\draw[->,decorate,decoration={snake,segment length=1.5mm,amplitude=0.3mm,post length = 2mm}] (01) -- (02);
\draw[->,decorate,decoration={snake,segment length=1.5mm,amplitude=0.3mm,post length = 2mm}] (14) -- (02);
\draw[->,decorate,decoration={snake,segment length=1.5mm,amplitude=0.3mm,post length = 2mm}] (33) -- (02);
\draw[->,decorate,decoration={snake,segment length=1.5mm,amplitude=0.3mm,post length = 2mm}] (-14) -- (02);
\draw[->,decorate,decoration={snake,segment length=1.5mm,amplitude=0.3mm,post length = 2mm}] (-34) -- (02);
\draw[->,decorate,decoration={snake,segment length=1.5mm,amplitude=0.3mm,post length = 2mm}] (14) -- (06);
\draw[->,decorate,decoration={snake,segment length=1.5mm,amplitude=0.3mm,post length = 2mm}] (14) -- (26);
\draw[->,decorate,decoration={snake,segment length=1.5mm,amplitude=0.3mm,post length = 2mm}] (-34) -- (-26);
\draw[->,decorate,decoration={snake,segment length=1.5mm,amplitude=0.3mm,post length = 2mm}] (-34) -- (-46);
\draw[->,decorate,decoration={snake,segment length=1.5mm,amplitude=0.3mm,post length = 2mm}] (-48) -- (-46);
\draw[->,decorate,decoration={snake,segment length=1.5mm,amplitude=0.3mm,post length = 2mm}] (-48) -- (-410);
\end{tikzpicture}
\caption{An absorption-free subuniverse $C\leq E_{-}(b)$.}\label{figure:AFsetC}
\end{figure}
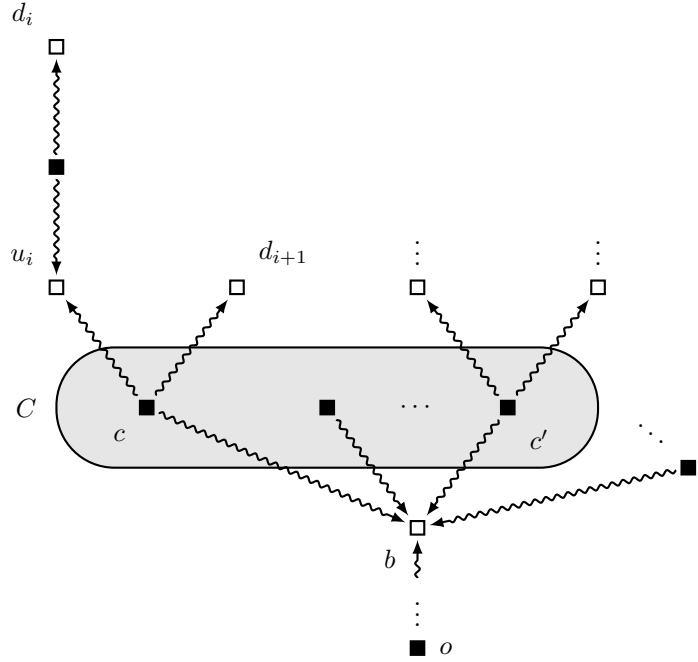

Note that $k_0<|H|$, and so the nonincreasing sequence $\langle k_0k_1\dots k_{|H|}\rangle$ of distances must stabilize. That is, there exists $i<|H|$ such that $k_i=k_{i+1}$. We denote this distance by $k$ and show in the following claim that it can only stabilize at $k=0$, that is, when $d_i=d_{i+1}=b$.

\begin{claim*}
The distance $k$ must be equal to zero.
\end{claim*}
\begin{subproof}[Proof of Claim]
Suppose for contradiction that $k\neq 0$ (and so $k\geq 2$, since $k$ is even). Pick any $c'\in C$.

First, recall that $C\leq E_-(b)$ which is definitely contained in $E_{k-1}(b)$. Therefore we have that $c'\in E_{k-1}(b)$. Second, note that $\dist_\mathbb T(d_i,c)$ is one less than $\dist_\mathbb T(d_i,b)=k$ (see Figure \ref{figure:AFsetC}). Therefore $c\in E_{k-1}(d_i)$. We combine these two facts with the compatibility of the relation $E$: Since $c\in E_{k-1}(d_i)$, $c'\in E_{k-1}(b)$, $d_i\mathbin{\circ}b=d_{i+1}$, and $\mathbin{\circ}$ preserves $E$, it follows that $c\mathbin{\circ}c'\in E_{k-1}(d_{i+1})$.

But we also have $c\mathbin{\circ}c'\in C$ (it is a subuniverse) and $E_{k-1}(d_{i+1})\cap C=\{c\}$. Thus we have proved that $c\mathbin{\circ}c'=c$ for all $c'\in C$, which means that $\{c\}\abs C$ via $\omega$, a contradiction with $C$ being absorption-free.
\end{subproof}

 We have proved that $k=0$, which means $d_i=d_{i+1}=b$ and thus by idempotency of $\mathbin{\circ}$ also $d_{|B|}=d\mathbin{\star}b=b$.
\end{proof}

Let us denote by $\mathcal F$ the smallest set of binary operations on $H$ satisfying
\begin{itemize}
\item $x\mathbin{\star}y\in\mathcal F$, $y\mathbin{\star}x\in\mathcal F$, and
\item if $\varphi(x,y)\in\mathcal F$, then $\{x\mathbin{\star}\varphi(x,y),y\mathbin{\star}\varphi(x,y),\varphi(x,y)\mathbin{\star}x,\varphi(x,y)\mathbin{\star}y\}\subseteq\mathcal F$,
\item if $\varphi(x,y),\varphi'(x,y)\in\mathcal F$, then $(\varphi(x,y)\star\varphi'(x,y))\in\mathcal F$.
\end{itemize}
From Lemma \ref{lemma:star_absorbs} and the construction of $\mathcal F$ we immediately obtain the following:

\begin{corollary} \label{corollary:F_absorbs}
If $d\in B$ is such that $c\prec d$ for some $c\in C$, then $\varphi(b,d)=\varphi(d,b)=b$ for every $\varphi\in\mathcal F$.
\end{corollary}

For every $c,c'\in C$ let $S_{c,c'}$ be the set $\{\varphi(c,c')\mid \varphi(x,y)\in\mathcal F\}\subseteq C$. We will use the following easy facts:
\begin{itemize}
\item $S_{c,c}=\{c\}$,
\item $S_{c,c'}=S_{c',c}$,
\item both $S_{c,c'}$ and $S_{c,c'}\cup\{c,c'\}$ are closed under the operation $\mathbin{\star}$, that is, they are subalgebras of the groupoid $(H;\star)$.
\item in particular, if $x,y\in S_{c,c'}$, then $S_{x,y}\subseteq S_{c,c'}$.
\end{itemize}
\begin{remark*}
Alternatively, using terminology from universal algebra, we could have defined $\mathcal F$ to be the set of all binary terms in the binary operation symbol $\star$ which contain both the variables $x$ and $y$. Then $S_{c,c'}$ would be the image of $\mathcal F$ under the homomorphism from the absolutely free two-generated algebra to $(C;\{\star\})$ given by $x\mapsto c$ and $y\mapsto c'$.
\end{remark*}

Note that $\mathcal F\subseteq\IdPol_2(\mathbb H)$. In the next lemma, we prove that, in fact, $\mathbb H$ has many more binary idempotent polymorphisms.

\begin{lemma} \label{lemma:binary_polymorphisms_on_C}
Let $\gamma:C^2\to C$ be any binary idempotent operation such that $\gamma(c,c')\in S_{c,c'}$ for all $c,c'\in C$. Then there exists $\tau\in\IdPol_2(\mathbb H)$ extending $\gamma$ (i.e., $\tau |_C=\gamma$).
\end{lemma}
\begin{proof}
Recall that $b\in B$ is such that $C\leq E_-(b)$, and $b\prec c$ for all $c\in C$. For every $c,c'\in C$ we fix some $\varphi_{c,c'}(x,y)\in\mathcal F$ such that $\varphi_{c,c'}(c,c')=\gamma(c,c')$. For $x,y\in H$ we define $\tau(x,y)$ in the following way:

\begin{enumerate}[(1)]
\item If $\lvl(x)=\lvl(y)$ and there exist $c,c'\in C$ such that the following two conditions hold:
\begin{enumerate}[(a)]
\item $b\prec x\prec c$ or $c\preceq x$, and
\item $b\prec y\prec c'$ or $c'\preceq y$,
\end{enumerate}
then we set $\tau(x,y)=\varphi_{c,c'}(x,y)$.
\item Else, we define $\tau(x,y)=x\mathbin{\star}y$.
\end{enumerate}
It follows immediately from the construction that $\tau$ is idempotent and $\tau |_C=\gamma$. To prove that $\tau$ is a polymorphism of $\mathbb H$, let $x\myrightarrow u$ and  $y\myrightarrow v$ be arbitrary edges in $\mathbb H$. We split the proof into separate arguments depending on the cases of the construction applied to $\tau(x,y)$ and $\tau(u,v)$:

\begin{enumerate}[I.]
  \item Both $\{x,y\}$ and $\{u,v\}$ fall under the same case.

  Then $\tau(x,y)\myrightarrow\tau(u,v)$ in $\mathbb H$ follows immediately from the fact that $\varphi_{c,c'}$ (in case (1)) and $\star$ (in case (2)) are polymorphisms of $\mathbb H$.

  \item $\{x,y\}$ falls under case (1) and $\{u,v\}$ falls under case (2).

  Because $x\myrightarrow u$ in $\mathbb H$ and $b$ lies on the top level of $\mathbb H$, if (1a) is true for $x$ then either it is true for $u$ as well (with the same $c\in C$), or $u=b$. (Indeed, if $b\in\Path_\mathbb H(o,x)$, then $b\in\Path_\mathbb H(o,u)$.) Similarly for $(1b)$, $y$ and $v$. Clearly $\lvl(u)=\lvl(v)$, so the only possible reason for $\{u,v\}$ not falling under case (1) is that $b\in\{u,v\}$.

  It follows that $\tau(u,v)=u\mathbin{\star}v=b=\varphi_{c,c'}(u,v)$, either by Corollary \ref{corollary:F_absorbs} or by idempotency (in case that $u=v=b$). We conclude that $\tau(x,y)\myrightarrow\tau(u,v)$ in $\mathbb H$ in this case as well.

  \item $\{x,y\}$ falls under case (2) and $\{u,v\}$ falls under case (1).

  Because $x\myrightarrow u$ in $\mathbb H$, $b$ lies on the top level of $\mathbb H$, and (1a) is true for $u$, it follows that (1a) must be true for $x$ as well. Similarly for $v$ and $y$. Therefore this situation cannot happen.\qedhere
\end{enumerate}
\end{proof}

As an easy consequence of this lemma, we can prove that $C$ has a binary idempotent commutative operation (i.e., a binary WNU operation).

\begin{corollary} \label{corollary:2wnu}
There exists $\varphi\in\IdPol_2(\mathbb H)$ such that $\varphi |_C$ is commutative.
\end{corollary}
\begin{proof}
For every $c,c'\in C$ define $\gamma(c,c')=\gamma(c',c)$ to be an arbitrary element from $S_{c,c'}$ thus making $\gamma$ commutative, and then apply Lemma \ref{lemma:binary_polymorphisms_on_C}.
\end{proof}

The above corollary implies that $|C|>2$, since a binary WNU operation on a 2-element set is a semilattice operation which would violate absorption-freeness. Unfortunately, a binary WNU operation is not enough to construct a weakly pointing operation for $C$; we need a slightly more involved argument.

\begin{lemma} \label{lemma:C_has_pointing}
$C$ has a weakly pointing operation.
\end{lemma}
\begin{proof}
We start by showing that every two-element set is weakly pointed to a singleton by some operation, with an additional ``symmetry'' (property (ii) in the following claim):
\begin{claim}\label{claim:C_has_pointing_claim_one}
For every $x,y\in C$ there exist $\varphi\in\IdPol(\mathbb H)$ (say it is $n$-ary), $z\in C$, $\mathbf{c^1},\dots,\mathbf{c^n}\in C^n$ and $\alpha:C\to C$ such that the following holds:
\begin{enumerate}[(i)]
\item $\varphi |_C$ weakly points $\{x,y\}$ to $\{z\}$ with witnessing tuples $\mathbf{c^1},\dots,\mathbf{c^n}$.
\item For every $i\in[n]$ and $u\in C$, $\varphi(c^i_1,c^i_2,\dots,c^i_{i-1},u,c^i_{i+1},\dots,c^i_n)=\alpha(u)$.
\end{enumerate}
\end{claim}
\begin{subproof}[Proof of Claim \ref{claim:C_has_pointing_claim_one}]
We will prove Claim \ref{claim:C_has_pointing_claim_one} by induction on $|S_{x,y}\cup\{x,y\}|$. Assume first that $S_{x,y}\cap\{x,y\}\neq\emptyset$, say $x\in S_{x,y}$ (the argument for $y\in S_{x,y}$ is analogous). In that case we can apply Lemma \ref{lemma:binary_polymorphisms_on_C} to construct $\varphi\in\IdPol_2(\mathbb H)$ such that $\varphi(x,y)=\varphi(y,x)=\varphi(x,x)=x$ and $\varphi |_C$ is commutative (see the proof of Corollary \ref{corollary:2wnu}). Claim \ref{claim:C_has_pointing_claim_one} follows since $\varphi |_C$ weakly points $\{x,y\}$ to $\{x\}$, the witnessing tuple is $(x,x)$ for both coordinates and $\alpha(u)=\varphi(u,x)$ for all $u\in C$. This also covers the base step of our induction (i.e., $S_{x,y}\subseteq\{x,y\}$).

We can now assume that $S_{x,y}\cap\{x,y\}=\emptyset$. Let us define $c=x\star y$, $x'=x\star c$ and $y'=y\star c$. Using Lemma \ref{lemma:binary_polymorphisms_on_C} we can construct $\varphi\in\IdPol_2(\mathbb H)$ such that $\varphi(x,c)=\varphi(c,x)=x'$ and $\varphi(y,c)=\varphi(c,y)=y'$ and $\varphi |_C$ is commutative. In particular, $\varphi |_C$ points $\{x,y\}$ to $\{x',y'\}$, the witnessing tuple is $(c,c)$ for both coordinates.

Since $x',y'\in S_{x,y}$, it follows that $S_{x',y'}\cup\{x',y'\}\subseteq S_{x,y}\subsetneq S_{x,y}\cup\{x,y\}$. Hence, by induction assumption, Claim \ref{claim:C_has_pointing_claim_one} holds for $x',y'$. Let it be witnessed by $\psi\in\IdPol(\mathbb H)$ (say $m$-ary) weakly pointing $\{x',y'\}$ to $\{z\}$ with witnessing tuples $\mathbf{d^1},\dots,\mathbf{d^m}\in C^m$ and let $\alpha':C\to C$ be the corresponding mapping from property (ii).

Using Observation \ref{observation:composing_pointing_operations} we get that $(\psi\compose\varphi)|_C$ weakly points $\{x,y\}$ to $\{z\}$. We will prove that (ii) holds as well, with $\alpha:C\to C$ given by $\alpha(u)=\alpha'(\varphi(u,c))$, for $u\in C$. Choose $i\in[m]$ and $j\in\{1,2\}$. The witnessing tuple for the $(2(i-1)+j)$th coordinate of $\psi\compose\varphi$ (given by the proof of Observation \ref{observation:composing_pointing_operations}) is
$$
\mathbf{c^{i,j}}=(d^i_1,d^i_1,d^i_2,d^i_2,\dots,d^i_{i-1},d^i_{i-1},c,c,d^i_{i+1},d^i_{i+1},\dots,d^i_m,d^i_m).
$$
We verify property (ii) of $\psi\compose\varphi$ for $j=1$ (the proof is the same for $j=2$, since $\varphi(c,u)=\varphi(u,c)$):
\begin{align*}
  &(\psi\compose\varphi)(d^i_1,d^i_1,\dots,d^i_{i-1},d^i_{i-1},u,c,d^i_{i+1},d^i_{i+1},\dots,d^i_m,d^i_m)\\
  =&\psi(\varphi(d^i_1,d^i_1),\dots,\varphi(d^i_{i-1},d^i_{i-1}),\varphi(u,c),\varphi(d^i_{i+1},d^i_{i+1}),\dots,\varphi(d^i_m,d^i_m))\\
  =&\psi(d^i_1,\dots,d^i_{i-1},\varphi(u,c),d^i_{i+1},\dots,d^i_m)\\
  =&\alpha'(\varphi(u,c))=\alpha(u).
\end{align*}
The first equation is just the definition of $\psi\compose\varphi$, the second equation follows from idempotency of $\varphi$, the third equation is property (ii) for $\psi$, and the last equation is the definition of $\alpha$.
\end{subproof}

We will now compose the operations from Claim \ref{claim:C_has_pointing_claim_one} to construct a weakly pointing operation for $C$; we use another induction argument.

\begin{claim}\label{claim:C_has_pointing_claim_two}
For every nonempty $X\subseteq C$ there exists $c\in C$ and $\varphi\in\IdPol(\mathbb H)$ such that $\varphi |_C$ weakly points $X$ to $\{c\}$.
\end{claim}
\begin{subproof}[Proof of Claim \ref{claim:C_has_pointing_claim_two}]
We prove Claim \ref{claim:C_has_pointing_claim_two} by induction on $|X|$. If $X=\{x\}$, then the claim is trivial: take any $\varphi\in\IdPol(\mathbb H)$, $c=x$ and witnessing tuple $(x,x,\dots,x)$ for all coordinates.

Let $|X|=k>1$ and assume that Claim \ref{claim:C_has_pointing_claim_two} holds for all at most $(k-1)$-element subsets of $C$.
 Pick any $u,v\in X$, $u\neq v$ and let $\varphi\in\IdPol(\mathbb H)$ (say $n$-ary), $w\in C$ and $\alpha:C\to C$ be the objects given Claim \ref{claim:C_has_pointing_claim_one} applied to $u$ and $v$. It is easy to see that $\varphi |_C$ weakly points $X$ to $Y=\{\alpha(x)\mid x\in X\}$
 (this is where we need property (ii) from Claim \ref{claim:C_has_pointing_claim_one}). Since $\alpha(u)=w=\alpha(v)$, it follows that $|Y|<|X|$. By induction assumption, there exists $\psi\in\IdPol(\mathbb H)$ and $c\in C$ such that $\psi|_C$ weakly points $Y$ to $\{c\}$. Using Observation \ref{observation:composing_pointing_operations} we get that $(\psi\compose\varphi)|_C$ weakly points $X$ to $\{c\}$ which concludes the proof of Claim~\ref{claim:C_has_pointing_claim_two}.
\end{subproof}
The lemma now follows from Claim \ref{claim:C_has_pointing_claim_two} applied to $X=C$.
\end{proof}

We have achieved the goal of this subsection, i.e., the following corollary.

\begin{corollary} \label{corollary:E-neighbourhoods_of_singletons_are_SDmeet}
For every $b\in B$, $E_-(b)$ is $\mathrm{SD}(\wedge)$. Similarly for $a\in A$ and $E_+(a)$.
\end{corollary}
\begin{proof}
By Lemma \ref{lemma:C_has_pointing}, every absorption-free subuniverse $C\leq E_-(b)$ has a weakly pointing operation and so we can apply Corollary \ref{corollary:every_AF_has_pointing}. The proof for $a\in A$ is analogous.
\end{proof}

\subsection{All absorption-free subuniverses are $\mathrm{SD}(\wedge)$}\label{subsection:all_AF}

The last step of our proof is to show that every absorption-free subuniverse $C$ of $A$ or $B$ has a weakly pointing operation. Theorem \ref{theorem:main} will then follow from Corollary \ref{corollary:every_AF_has_pointing} and Corollary \ref{corollary:reduction_to_top_and_bottom}.

\begin{lemma} \label{lemma:every_AF_of_A_has_pointing}
Every absorption-free subuniverse $C$ of $A$ or $B$ has a weakly pointing operation.
\end{lemma}
\begin{proof}
Recall that by Corollary \ref{corollary:AF_has_single_distance}, for every absorption-free subuniverse $C$ of $A$ or $B$ there exists $k>0$ such that $\dist_\mathbb T(c,o)=k$ for all $c\in C$. We will proceed by induction on this distance $k$. The base step, $k=1$, follows from Lemma \ref{lemma:C_has_pointing} from the previous subsection, since in that case $C\leq E_+(o)$.

Let $k>1$ and assume that $C\leq A$ (the proof for $C\leq B$ is analogous). Let us denote by $D$ the subuniverse $D=E_+(C)\cap E_{k-1}(o)\leq B$. If $D=\{d\}$ for some $d\in B$, then $C\leq E_-(d)$ and $C$ has a weakly pointing operation by Lemma \ref{lemma:C_has_pointing}. Thus we can assume that $|D|>1$.

The binary relation $E\cap (C\times D)$ induces an onto mapping $\eta:C\to D$ defined by $\eta(c)=d$, where $d\in D$ is unique such that $(c,d)\in E$ (this is because $\mathbb T$ is a tree; see Figure \ref{figure:AF_CandD}).

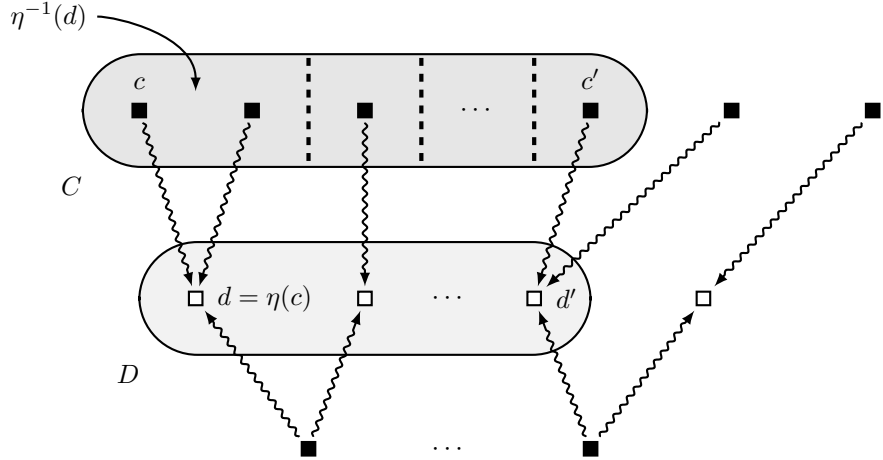
\begin{figure}[ht!]
\centering
\begin{tikzpicture}[thick,xscale=1.5,yscale=2.5,
  blacksquare/.style={rectangle,draw,outer sep=2pt,inner sep=0pt,minimum size=1.8mm,fill},
  whitesquare/.style={rectangle,draw,outer sep=2pt,inner sep=0pt,minimum size=1.8mm,fill=white},
  >=latex]
\draw[rounded corners=22pt,fill=gray!20!white] (2.5,0.3) rectangle (-2.5,-0.3) {};
\node (C) at (-2.6,-0.4) {$C$};
\draw[rounded corners=22pt,fill=gray!10!white] (-2,-1.3) rectangle (2,-0.7) {};
\node (D) at (-2.1,-1.4) {$D$};
\node [blacksquare,label=north:{$c$}] (c1) at (-2,0) {};
\node [blacksquare](c2) at (-1,0) {};
\node [blacksquare](c3) at (0,0) {};
\node [blacksquare,label=north:{$c'$}](c4) at (2,0) {};
\node [blacksquare](c5) at (3.25,0) {};
\node [blacksquare](a1) at (4.5,0) {};
\node (cdots) at (1,0) {$\cdots$};
\node [whitesquare,label=east:{$d=\eta(c)$}](d1) at (-1.5,-1) {};
\node [whitesquare](d2) at (0,-1) {};
\node [whitesquare,label=east:{$d'$}](d3) at (1.5,-1) {};
\node [whitesquare](b1) at (3,-1) {};
\node (cdots) at (0.75,-1) {$\cdots$};
\draw[->,decorate,decoration={snake,segment length=1.5mm,amplitude=0.3mm,post length = 2mm}] (c1) -- (d1);
\draw[->,decorate,decoration={snake,segment length=1.5mm,amplitude=0.3mm,post length = 2mm}] (c2) -- (d1);
\draw[->,decorate,decoration={snake,segment length=1.5mm,amplitude=0.3mm,post length = 2mm}] (c3) -- (d2);
\draw[->,decorate,decoration={snake,segment length=1.5mm,amplitude=0.3mm,post length = 2mm}] (c4) -- (d3);
\draw[->,decorate,decoration={snake,segment length=1.5mm,amplitude=0.3mm,post length = 2mm}] (c5) -- (d3);
\draw[->,decorate,decoration={snake,segment length=1.5mm,amplitude=0.3mm,post length = 2mm}] (a1) -- (b1);
\node [blacksquare](e1) at (-0.5,-1.8) {};
\node [blacksquare](e2) at (2,-1.8) {};
\node (cdots) at (0.75,-1.8) {$\cdots$};
\draw[->,decorate,decoration={snake,segment length=1.5mm,amplitude=0.3mm,post length = 2mm}] (e1) -- (d1);
\draw[->,decorate,decoration={snake,segment length=1.5mm,amplitude=0.3mm,post length = 2mm}] (e1) -- (d2);
\draw[->,decorate,decoration={snake,segment length=1.5mm,amplitude=0.3mm,post length = 2mm}] (e2) -- (d3);
\draw[->,decorate,decoration={snake,segment length=1.5mm,amplitude=0.3mm,post length = 2mm}] (e2) -- (b1);
\draw[dashed,line width=0.5mm,ultra thick] (-0.5,0.28) -- (-0.5,-0.28);
\draw[dashed,line width=0.5mm,ultra thick] (0.5,0.28) -- (0.5,-0.28);
\draw[dashed,line width=0.5mm,ultra thick] (1.5,0.28) -- (1.5,-0.28);
\node (label) at (-2.8,0.5) {$\eta^{-1}(d)$};
\node (target) at (-1.5,0.05) {};
\draw[->] (label) .. controls (-2.2,0.5) and (-1.5,0.4) .. (target);
\end{tikzpicture}
\caption{The absorption-free subuniverses $C$ and $D$.}\label{figure:AF_CandD}
\end{figure}

The relation  $E\cap (C\times D)$ is preserved by every $\varphi\in\IdPol(\mathbb H)$ (see Lemma \ref{lemma:ABE_subuniverses}). The following are easy consequences of this fact:
\begin{itemize}
\item for every $D'\leq D$ the set $\eta^{-1}(D')$ is a subuniverse of $C$,
\item if $D'\abs D$, then $\eta^{-1}(D')\abs C$ (the absorbing polymorphism is the same),
\item for $D'\leq D$, $D'=D$ if and only if $\eta^{-1}(D')=C$ (since $\eta$ is onto).
\end{itemize}
Combining these facts together with the fact that $C$ is absorption-free yields that $D$ is absorption-free. Hence by induction assumption $D$ has a weakly pointing operation.

Let $\varphi\in\IdPol(\mathbb H)$ (say $n$-ary) be such that $\varphi|_D$ weakly points $D$ to $\{d\}$ with witnessing tuples $\mathbf{d^1},\dots,\mathbf{d^n}$. It is easy to verify that $\varphi|_C$ weakly points $C$ to $\eta^{-1}(d)$; any $\mathbf{c^1},\dots,\mathbf{c^n}\in C^n$ such that $\eta(c^i_j)=d^i_j$ (for $i,j\in[n]$) can serve as witnessing tuples.

Since $\eta^{-1}(d)\leq E_-(d)$, it follows from Corollary \ref{corollary:E-neighbourhoods_of_singletons_are_SDmeet} and Theorem \ref{theorem:every_subuniverse_has_pointing} that $\eta^{-1}(d)$ has a weakly pointing operation. Let $\psi\in\IdPol(\mathbb H)$ and $c\in\eta^{-1}(d)$ be such that $\psi|_{\eta^{-1}(d)}$ weakly points $\eta^{-1}(d)$ to $\{c\}$. In particular, $\psi|_C$ weakly points $\eta^{-1}(d)$ to $\{c\}$ and thus by Observation \ref{observation:composing_pointing_operations}, $(\psi\compose\varphi)|_C$ weakly points $C$ to $c$.
\end{proof}

\begin{remark*}
In the language of universal algebra, the relation $E\cap (C\times D)$ is the graph of an onto homomorphism $\eta:C\to D$ and thus, by the First Isomorphism Theorem, $D$ is isomorphic to the quotient of $C$ over the kernel of $\eta$. The induction step in the previous lemma follows easily from this observation.
\end{remark*}

\begin{proof}[Proof of Theorem \ref{theorem:main} and Corollary \ref{corollary:dichotomy}]
Let $\mathbb H$ be a special tree such that $\alg\mathbb H$ is Taylor. In Lemma \ref{lemma:every_AF_of_A_has_pointing} we proved that every absorption-free subuniverse of $A$ or $B$ has a weakly pointing operation. By Corollary \ref{corollary:every_AF_has_pointing}, both $A$ and $B$ are $\mathrm{SD}(\wedge)$ and thus it follows from Corollary \ref{corollary:reduction_to_top_and_bottom} that $\alg\mathbb H$ is $\mathrm{SD}(\wedge)$.

It is easy to see that the core of a special tree is again a special tree. If $\mathbb H$ is a core, then either $\alg\mathbb H$ is not Taylor, in which case $\mathrm{CSP}(\mathbb H)$ is \textbf{NP}-complete by Theorem \ref{theorem:noWNU_NPc}, or $\alg\mathbb H$ is $\mathrm{SD}(\wedge)$ and $\mathbb H$ has bounded width by Theorem~\ref{theorem:bounded_width}.
\end{proof}

\section{Discussion}

We believe that given the evidence, it is reasonable to conjecture that our result generalizes to all oriented trees. (A natural first step would be to confirm this conjecture for all triads.) Moreover, we hope that the techniques developed in this paper will be useful in pursuit of the proof.

\begin{conjecture}
For every oriented tree $\mathbb H$, either $\alg\mathbb H$ is not Taylor or it is $\mathrm{SD}(\wedge)$. In particular, if $\mathbb H$ is a core, then $\mathbb H$ has bounded width or $\mathrm{CSP}(\mathbb H)$ is \textbf{NP}-complete.
\end{conjecture}

The reader may wonder why we need two different characterizations of $\mathrm{SD}(\wedge)$ algebras, i.e., why we use WNU operations for the proof of Corollary \ref{corollary:reduction_to_top_and_bottom}. The reason is that our techniques used later in the proof are not well suited to deal with non-diagonal connected components of powers of the digraph $\mathbb H$. This is one obstacle to generalizing the result to all oriented trees.

Another shortcoming is that we cannot get a good handle of polymorphisms of higher arities than binary. For example, it follows from Corollary \ref{corollary:2wnu} that neither $A$ nor $B$ can have a two-element absorption-free subuniverse, and, in fact, we can prove that $\alg\mathbb H$ (if it is Taylor) cannot have a two-element absorption-free subuniverse at all (we will not present the argument here, but it is similar in spirit to the proof of Lemma \ref{lemma:extending_WNUs}). We do not know if this result can be extended to more than two elements. Hence the following open problem.

A finite idempotent algebra $\mathbf A$ is \emph{always absorbing}, if for every nonempty $B\leq\mathbf A$ there exists $b\in B$ such that $\{b\}\abs B$. (Equivalently, there are no absorption-free algebras in the pseudovariety generated by $\mathbf A$, see \cite[Proposition 2.1]{barto_maltsev_2013}.)

\begin{problem*}
Let $\mathbb H$ be a (special, or any oriented) tree such that $\alg\mathbb H$ is Taylor. Is $\alg\mathbb H$ \emph{always absorbing}?
\end{problem*}

\noindent By Corollary \ref{corollary:every_AF_has_pointing}, always absorbing algebras are $\mathrm{SD}(\wedge)$. A positive answer to this problem would likely significantly simplify our proof.

Lemmata \ref{lemma:WNUabsorbing_E2o} and \ref{lemma:singleton_absorbing} establish existence of a singleton absorbing subuniverse of $A$ or $B$. Incidentally, the Absorption Theorem of Barto and Kozik \cite[Theorem 2.3]{barto_absorbing_2012} applied to $A$, $B$,
 and $E$ immediately yields that fact. However, for the rest of the proof we need the fact that the absorbing operation is a WNU operation. Can the proof of the Absorption Theorem in \cite{barto_absorbing_2012} be improved to achieve this stronger claim?

\emph{Special balanced digraphs}, a relaxation of the definition of special trees to balanced digraphs, appear naturally in the reduction of constraint satisfaction problems to digraph $\mathbb H$-coloring \cite{bulin_reduction_2013,bulin_finer_2015,feder_computational_1999}. The reader may notice similarities with some of the proofs in \cite{bulin_finer_2015}. Can our techniques be adapted to obtain interesting results about special balanced digraphs?

\section*{Acknowledgements}
The author would like to thank Libor Barto and the anonymous reviewers for thoughtful comments and valuable input. The work was supported by the following grant projects: GA\v CR
  201/09/H012 \& 13-01832S, M{\v S}MT {\v C}R 7AMB13PL013, GA UK 67410 \& 558313, and SVV-2014-260107.

\bibliography{spectrees}
\bibliographystyle{spmpsci}

\end{document}